\documentclass{amsart}

\usepackage[letterpaper,hmargin=1.5in,vmargin=0.8in]{geometry}

\usepackage{amssymb, amstext, amsthm, amscd, amsmath,color}
\usepackage{mathtools}
\usepackage{hyperref}
\usepackage{xcolor}
\definecolor{darkgreen}{rgb}{0,0.45,0}
\hypersetup{colorlinks,urlcolor=blue,citecolor=darkgreen,linkcolor=darkgreen,linktocpage}
\usepackage{comment}
\usepackage[capitalize]{cleveref}
\usepackage{thm-restate}
\usepackage{mathbbol}
\usepackage{float}

\usepackage{ stmaryrd }

\usepackage[shortlabels]{enumitem}
\setlist{leftmargin=8mm}

\usepackage{mathrsfs}
\usepackage{tikz-cd}

\usepackage{graphicx}
\graphicspath{ {./images/} }

\newcommand{\id}{\mathrm{id}}

\newcommand{\ca}{\mathcal{A}}
\newcommand{\cb}{\mathcal{B}}
\newcommand{\cc}{\mathcal{C}}

\newcommand{\ce}{\mathcal{E}}
\newcommand{\cf}{\mathcal{F}}
\newcommand{\cg}{\mathcal{G}}

\newcommand{\ct}{\mathcal{T}}

\newcommand{\cp}{\mathcal{P}}
\newcommand{\cq}{\mathcal{Q}}

\newcommand{\cx}{\mathcal{X}}
\newcommand{\cy}{\mathcal{Y}}

\newcommand{\bgam}{\Gamma}

\newcommand{\blam}{\Lambda}

\newcommand{\N}{{\mathbb N}}

\newcommand{\R}{{\mathbb R}}
\newcommand{\Z}{{\mathbb Z}}

\newcommand{\K}{\mathbb{k}}

\newcommand{\supp}{\mathrm{supp}}

\DeclareMathOperator*{\coker}{coker}

\DeclareMathOperator{\Hom}{Hom}
\DeclareMathOperator{\End}{End}

\DeclareMathOperator{\fmod}{mod}
\newcommand{\kvec}{{\mathrm{vec}_\K}}
\newcommand{\kVec}{{\mathrm{Vec}_\K}}

\newcommand{\Ab}{\mathrm{Ab}}

\DeclareMathOperator{\rep}{rep}
\newcommand{\repp}{{\rep \cp}}
\newcommand{\repq}{{\rep \cq}}

\newcommand{\op}{\mathrm{op}}

\newcommand{\proj}[1]{{\mathrm{proj} (#1)}}
\newcommand{\inj}[1]{{\mathrm{inj} (#1)}}
\newcommand{\add}[1]{{\mathrm{add} (#1)}}
\DeclareMathOperator{\spread}{sprd}
\newcommand{\addspread}[1]{\add{\spread #1}}

\newcommand{\acx}{{\mathrm{add}(\mathcal{X})}}
\newcommand{\acy}{{\mathrm{add}(\mathcal{Y})}}

\DeclareMathOperator{\rad}{rad}

\newcommand{\gldim}[2]{\mathrm{gl.dim}_{#1}\!\left(#2\right)}
\newcommand{\pdim}[2]{\mathrm{p.dim}_{#1}\!\left(#2\right)}

\newcommand{\floor}[1]{{\lfloor #1 \rfloor}}
\newcommand{\Ran}{\mathsf{Ran}}
\newcommand{\Lan}{\mathsf{Lan}}

\newcommand{\res}{\mathsf{Res}}
\DeclareMathOperator{\pad}{pad}

\DeclareMathOperator*{\colim}{colim}

\newcommand{\onemorphism}{\mathbb{1}}

\DeclareMathOperator{\cover}{cov}
\DeclareMathOperator{\cocover}{cocov}
\DeclareMathOperator{\incomp}{incomp}
\DeclareMathOperator{\minimal}{min}
\DeclareMathOperator{\maximal}{max}

\newcommand{\ds}[1]{{{\downarrow}#1}}
\newcommand{\us}[1]{{{\uparrow}#1}}
\newcommand{\iprod}[2]{{\langle #1, #2 \rangle}}
\newcommand{\dsminus}[2]{{{\downarrow}#1 \setminus {\downarrow}#2}}% > D, x>
\newcommand{\usminus}[2]{{{\uparrow}#1 \setminus {\uparrow}#2}}

\renewcommand{\phi}{\varphi}
\renewcommand{\epsilon}{\varepsilon}

\theoremstyle{plain}
\newtheorem{thm}{Theorem}[section]
\newtheorem{cor}[thm]{Corollary}
\newtheorem{prop}[thm]{Proposition}
\newtheorem{lem}[thm]{Lemma}

\newtheorem{question}{Question}

\theoremstyle{definition}
\newtheorem{defn}[thm]{Definition}
\newtheorem{eg}[thm]{Example}
\newtheorem{rem}[thm]{Remark}
\newtheorem{notation}[thm]{Notation}

\newtheorem{assumptions}[thm]{Assumptions}

\newtheorem{theoremx}{Theorem}

\newtheorem{corollaryx}[theoremx]{Corollary}

\newtheorem{conjecture}{Conjecture}

\Crefname{prop}{Proposition}{Propositions}
\Crefname{cor}{Corollary}{Corollaries}
\Crefname{lem}{Lemma}{Lemmas}
\Crefname{thm}{Theorem}{Theorems}
\Crefname{defn}{Definition}{Definitions}
\Crefname{rem}{Remark}{Remarks}
\Crefname{eg}{Example}{Examples}
\Crefname{constr}{Construction}{Constructions}
\Crefname{assumptions}{Assumptions}{Assumptions}
\Crefname{corollaryx}{Corollary}{Corollaries}
\Crefname{theoremx}{Theorem}{Theorems}

\def\noteson{%
\gdef\luis##1{\noindent{\color{blue}[Luis: ##1]}}%
\gdef\justin##1{\noindent{\color{red}[Justin: ##1]}}%
\gdef\ben##1{\noindent{\color{darkgreen}[Ben: ##1]}}%
\gdef\eric##1{\noindent{\color{orange}[Eric: ##1]}}%
\gdef\todo##1{\noindent{\color{violet}[to do: ##1]}}}

\noteson

\title{Stabilization of the spread-global dimension}

\author{Benjamin Blanchette}
\address{Département de mathématiques, Université de Sherbrooke, Québec, Canada}

\author{Justin Desrochers}
\address{Département de mathématiques, Université de Sherbrooke, Québec, Canada}
\author{Eric J. Hanson}
\address{Department of Mathematics, North Carolina State University, Raleigh, NC 27695, USA}
\author{Luis Scoccola}
\address{Centre de Recherches Mathématiques et Institut des sciences mathématiques;
Laboratoire de combinatoire et d'informatique mathématique de l'Université du Québec à Montréal;
Université de Sherbrooke; Québec, Canada}

\begin{document}

\maketitle

\begin{abstract}
Motivated by constructions from applied topology, there has been recent interest in the homological algebra of linear representations of posets, particularly in the context of homological algebra relative to non-standard exact structures.
A prominent example is the spread exact structure on the category of representations of a fixed poset, in which the indecomposable projectives are the spread representations (that is, the indicator representations of convex and connected subsets).
The spread-global dimension is known to be finite for finite posets and not uniformly bounded on the collection of all Cartesian products between two arbitrary finite total orders.
It was conjectured in [AENY23] that the spread-global dimension is uniformly bounded on the collection of all Cartesian products between a fixed finite total order and an arbitrary finite total order.
We provide a positive answer to this conjecture and, more generally, prove that the spread-global dimension is uniformly bounded on the collection of all Cartesian products between a fixed finite poset and an arbitrary finite total order.
In doing so, we also establish the existence of finite spread-resolutions for finitely presented representations of arbitrary grid posets.
\end{abstract}

\section{Introduction}

\subsection{Context and problem statement}
Let $\K$ be any field, fixed throughout this paper.

\smallskip\noindent\textbf{Poset representations.}
A linear representation of a poset $\cp$ is a functor $\cp \to \kVec$.
The study of linear representations of posets has a long history.
Of particular importance is the work of Nazarova and Roiter on matrix representations of posets~\cite{nazarova-roiter}, leading to the solution to the second Brauer--Thrall conjecture for finite-dimensional algebras \cite{nazarova-roiter-2},
and the work of Gabriel \cite{gabriel} on quivers of finite type, which relies on filtered representations of posets.
Both matrix representations and filtered representations are special cases of linear representations of posets in the above sense; for a treatment from this perspective, see \cite{simson}.
The classification of finite posets according to their representation type, which is significantly more involved than the case of quivers, is in
\cite{loupias,zavadskij,drozdowski,leszczynski1,leszczynski2}.
The homological properties of poset representations, such as global dimension, are also of interest \cite{mitchell-2,mitchell-3,baclawski,changchang}, and these show up, for example, in connection with standard cohomology of simplicial complexes \cite{gerstenhaber-schack,igusa-zacharia}.
%For $\cp$ a finite poset, the category of linear representations of $\cp$ is equivalent to the category of modules over the incidence algebra $\K \cp$.

\smallskip\noindent\textbf{Poset representations in applied topology.}
More recently, there has been interest in the study of poset representations coming from \emph{persistence theory} \cite{oudot,botnan-lesnick}, an area grown out of applied topology and Morse theory.
In this context, linear representations of a poset $\cp$ show up as the homology with field coefficients of $\cp$-filtered topological spaces, that is, families of topological spaces $\{X_p\}_{p \in \cp}$ with $X_p \subseteq X_q$ when $p \leq q$, as in Morse theory \cite{bauer-medina-schmahl}.
When the poset $\cp$ is not a total order, its category of representations is typically of wild representation type, meaning that its indecomposable representations cannot be classified effectively, and that these representations ought to be studied by means of incomplete invariants.
A recent line of work in persistence theory uses invariants derived from homological constructions, and, in particular, classical Betti tables \cite{lesnick-wright,OS24} and Betti tables with respect to non-standard exact structures \cite{boo,AENY23,BBH2022,botnan-oppermann-oudot-scoccola,
chacholski-guidolin-ren-scolamiero-tombari,asashiba24}.
% in the sense of \cite{draxler-reiten-smalo-solberg}.

\smallskip\noindent\textbf{Exact structures.}
An \emph{exact structure} on an Abelian category $\ca$ consists of a class of short exact sequences of $\ca$ satisfying certain conditions \cite{draxler-reiten-smalo-solberg}.
An exact structure~$\ce$ defines a notion of $\ce$-exactness, and hence notions of $\ce$-projective object, $\ce$-projective resolution, and $\ce$-global dimension,
and is useful in a variety of settings
\cite{hochschild,eilenber-moore,auslander-solberg,draxler-reiten-smalo-solberg}.
In this paper, we are concerned with exact structures of the form
$\ce = \cf_\cx$, where $\cx$ is a class of objects of $\ca$, and $\cf_\cx$ consists of all short exact sequences of $\ca$ that remain exact after applying $\Hom_\ca(X,-) : \ca \to \Ab$ for every $X \in \cx$.
The class $\cf_\cx$ is an exact structure for every collection of objects $\cx$
\cite[Prop.~1.7]{draxler-reiten-smalo-solberg},
and the following are standard questions about $\cf_\cx$:
\begin{enumerate}
    %\item What are the projective objects of $\cf_\cx$?
    \item[(Q1)] Are the $\cf_\cx$-projective objects the same as the $\cx$-decomposable objects?
    \item[(Q2)] Does every object of $\ca$ admit a finite $\cf_\cx$-projective resolution?
    \item[(Q3)] Is the $\cf_\cx$-global dimension finite?
\end{enumerate}

In persistence theory,
the category $\ca$ is often $\rep \cp$, the category of finitely presented representations of a poset $\cp$,
and the poset $\cp$ is either a finite poset or an upper semilattice, most commonly a grid poset (i.e., a finite Cartesian product of total orders) such as $\R^n$ \cite{botnan-lesnick}.
The exact structures most commonly used are the
\emph{standard exact structure},
where $\cx$ consists of the representable functors,
the \emph{rank exact structure} \cite{boo},
where $\cx$ consists of the representable functors 
together with the quotients of two representable functors,
and the \emph{spread exact structure} \cite{AENY23},
where $\cx$ consists of all \emph{spread representations}, recalled below.

In the case of the standard exact structure on $\rep \cp$ with $\cp$ finite or an upper semilattice, 
there are classical positive answers to the above questions:
(Q1) the projective representations are the direct sums of representable functors,
(Q2) every object admits a finite projective resolution (see \cite[Cor.~10.6]{chacholski-jin-tombari} for the upper semilattice case)
and (Q3) the global dimension is finite when $\cp$ is finite or a grid poset.
In fact, when $\cp = \R^n$, the global dimension is $n$, by Hilbert's syzygy theorem; see also \cite{iyama-marczinzik}, where other lattices are considered.

In the case of the rank exact structure, positive answers to (Q1) and (Q2) are given in \cite{boo}; a positive answer to (Q3) in the case of finite posets is given in \cite{BBH2022}, and in the case of upper semilattices in
\cite{botnan-oppermann-oudot-scoccola,chacholski-guidolin-ren-scolamiero-tombari};
in particular \cite{botnan-oppermann-oudot-scoccola} shows that the global dimension is $2n-2$ when $\cp = \R^n$.

Due to its complexity, less is known about the spread exact structure.
We now describe some recent advances; see also \cite{AENY23,aoki-escolar-tada,BBH2023}.
%, where various useful results about exact structures on categories of representations of infinite posets are proven.

\smallskip\noindent\textbf{The spread exact structure.}
A \emph{spread}%
\footnote{Various references in persistence theory call spreads ``intervals''; we prefer the name spread \cite{BBH2023} since it avoids confusion with the standard notion of interval from the theory of posets and combinatorics.}
of a poset~$\cp$ is a subset $S \subseteq \cp$ that is poset-connected (non-empty and such that any two elements in $S$ are connected by a zigzag of comparable elements in $S$) and poset-convex (for any pair of elements in $S$, the closed segment between them is also in $S$), and its corresponding \emph{spread representation}
%$\K_S : \cp \to \kvec$
is the indicator representation of $S$, that is,
the representation which takes the value $\K$ on $S$ and zero elsewhere, with all morphisms that are not forced to be zero being the identity of $\K$.
Spread representations were considered very early in the representation theory of posets; for example \cite{chaptal} classifies finite posets whose indecomposable representations are spread representations.
%The \emph{spread exact structure} is the exact structure whose projective representations are the direct sums of spread representations.
%; this exact structure does not always exist, but it does for example when~$\cp$ is finite (thanks to \cite[Thm.~1.15]{auslander-solberg} and the fact that, in a finite poset, there are finitely many spreads).

Questions (Q1), (Q2), and (Q3) are essentially open in the case of $\cp$ infinite and $\cf_\cx$ the spread exact structure.
In this setting, and to our knowledge, it is only known that 
(Q3) has a negative answer for $\cp = \R^n$, with $n \geq 2$ \cite[Ex.~7.24(3)]{BBH2023}.
However, a positive answer for (Q2) would still allow one to define invariants using spread-resolutions as done, in the case of a finite poset, in
\cite{BBH2022,AENY23,escolar-kim,brustle-et-al}.

If $\cp$ is finite, then there are finitely many spreads, and thus the projective objects of the spread exact structure are the spread-decomposable representations,
by a standard result \cite[Thm.~1.15]{auslander-solberg}; this gives a positive answer to (Q1) in this case.
Positive answers to (Q2) and (Q3) are given in \cite{AENY23} also for finite posets, but no explicit bound on the global dimension is currently known.
%In this direction, still in the finite poset case,
%it was recently proven that the spread-global dimension is monotonic with respect to inclusions of posets \cite[Thm.~1.2]{aoki-escolar-tada} (whenever $\cq \subseteq \cp$ is a (full) inclusion of posets, the spread-global dimension of $\cq$ is bounded above by the spread-global dimension of $\cp$).
%This is interesting since the standard global dimension does not have this property
%\cite[Sec.~3]{igusa-zacharia}.

The spread-global dimension of a two-dimensional grid poset $\cp = [k] \times [k]$ is not uniformly bounded as $k \to \infty$ \cite[Ex.~7.24(3)]{BBH2023}, where
$[k] = \{0, \dots, k-1\}$ denotes the standard $k$-element totally ordered set.
%\justin{we use $\ct$ for total order in the later sections}
%The following facts have been established in the literature:
%\begin{itemize}
%\item
%The spread exact structure has finite global dimension for every finite poset~\cite{AENY23}.
%\item
%This implies that the spread-global dimension of infinite posets that are standard in persistence theory, such as $\Z^2$ and $\R^2$, cannot be finite, in contrast to the standard and rank-global dimension of such posets.
However, it is conjectured in \cite{AENY23} that, if one of the two dimensions of the two-dimensional grid is kept fixed, then the spread-global dimension stabilizes:

\begin{conjecture}[{\cite[Conj.~4.11]{AENY23}}]
    \label{conjecture:1}
    If $k \geq 4$, then $\gldim{\spread}{ [k] \times [2] } = 2$.
\end{conjecture}

\begin{conjecture}[{\cite[Conj.~4.12]{AENY23}}]
    \label{conjecture:2}
    For every $m \geq 1 \in \mathbb{N}$,
    the function $g_m : \mathbb{N} \to \mathbb{N}$ given by $g_m(k) = \gldim{\spread}{ [k] \times [m] }$ is constant when $k \geq 2 + m$.
\end{conjecture}

\cref{conjecture:1} is relevant, for example, in the study of morphisms between representations of totally ordered sets~\cite{escolar-hiraoka,asashiba-escolar-hiraoka-takeuchi,jacquard-nanda-tillmann,urbancic-giansiracusa},
since the data given by two representations $M, N \in \rep([k])$ together with a morphism $M \to N$ is equivalent to the data given by a single representation of $\rep([k] \times [2])$, sometimes known as a persistence module over a commutative ladder.

\subsection{Contributions}
Recall that, for $\cp$ a poset, we let $\rep \cp$ denote the category of finitely presented representations of $\cp$,
and note that, if $\cp$ is finite, this is equivalently the category of finite-dimensional representations.

Our first main result gives positive answers to (Q1) and (Q2) for the spread-exact structure on the category of finitely presented representations of arbitrary grid posets, and in particular $\R^n$.
Recall that a \emph{grid poset} is any finite Cartesian product of total orders.

\begin{restatable}{theoremx}{finiteresolutionsinfinitecase}
    \label{theorem:answers-1-2-lattice}
    Let $\cp$ be a grid poset.
    The projective representations of the spread exact structure on $\rep \cp$ are the spread-decomposable representations.
    Every finitely presented representation of $\cp$ admits a finite spread-resolution.
\end{restatable}

\cref{theorem:answers-1-2-lattice} is a consequence of the following result, which allows one to construct spread-resolutions of finitely presented representations over arbitrary grid posets by restricting them to a finite subposet, as is the case for the standard and for the rank exact structure.
In the result, a \emph{spread-approximation} of a representation $M$ is a morphism $C \to M$ with $C$ spread-decomposable, and such that every other morphism from a spread-decomposable representations into $M$ factors through it.
An \emph{aligned grid inclusion} is a poset morphism between grid posets that is the Cartesian product of inclusions of totally ordered sets; see \cref{section:aligned-grid-inclusion}.

\begin{restatable}{propositionx}{mainextensiontheorem}
    \label{theorem:main-extension-theorem}
    Let $\iota : \cq \to \cp$ be an aligned grid inclusion, with $\cq$ finite.
    The restriction functor $\res_\iota : \rep{\cp} \to \repq$ has a left adjoint $\Lan_\iota : \repq \to \repp$, which is exact and fully faithful, and which maps spread-approximations to spread-approximations.
\end{restatable}

%\begin{theoremx}
%    \label{theorem:finite-spread-resolutions}
    %For every finitely presented representation $M$ of $\cp$, there exists a finite, join-closed subset $\cp' \subseteq \cp$ and a representation $M'$ of $\cp'$ such that $M$ is the left Kan extension of $M'$ along the inclusion $\cp' \to \cp$.
    %A finitely presented representation of $\cp$ is fp.spread-projective if and only if it decomposes as a direct sum of fp.spread representations.
    %\luis{I think we need to say that an fp.spread-resolution can be obtained by extending a spread-resolution over a finite upper semilattice}
%\end{theoremx}

Our second main result concerns (Q3), also in the case of the spread exact structure, and is an extension of the assertion in \cref{conjecture:2} to arbitrary finite posets.

\begin{restatable}{theoremx}{stabilization}
    \label{theorem:stabilization}
    For every finite poset $\cq$, we have $n_\cq < \infty$, where
    \[
        n_\cq \; \coloneqq \;\sup_{
            \substack{\ct \text{ finite}\\\text{total order}}}\,
            \gldim{\spread}{\ct \times \cq}.
    \]
    If $\cg$ is a finite grid poset, then $n_\cg = \gldim{\spread}{\left[k\right] \times \cg}$, with $k \coloneqq 1 + 4\cdot |\cg|$.
    %If $\cq = [k_1] \times \cdots \times [k_n]$, then $n_\cq  = \gldim{\spread}{[k] \times \cq}$, for $k = 1 + 4 \cdot k_1 \cdots k_n$.
\end{restatable}

Using \cref{theorem:stabilization} we give a positive answer to \cref{conjecture:1}:

\begin{corollaryx}[\cref{conjecture:1}]
    \label{corollary:conjecture-1}
    If $k \geq 4$, then $\gldim{\spread}{ [k] \times [2] } = 2$.
\end{corollaryx}

\cref{theorem:stabilization} also implies a slight weakening of \cref{conjecture:2}, where the stabilization constant is still linear in $m$.

\begin{corollaryx}[cf.~\cref{conjecture:2}]
    \label{corollary:conjecture-2}
    For every $m \geq 1 \in \mathbb{N}$,
    the function $g_m : \mathbb{N} \to \mathbb{N}$ given by $g_m(k) = \gldim{\spread}{ [k] \times [m] }$ is constant when $k \geq 1 + 4m$.
\end{corollaryx}

\cref{theorem:main-extension-theorem,theorem:stabilization} imply that certain infinite posets, such as $\R \times \cg$ with $\cg$ a finite grid poset, have finite spread-global dimension:

\begin{corollaryx}
    \label{corollary:R-times-lattice}
    If $\cg$ is a finite grid and $\ct$ is a total order, then
    $\gldim{\spread}{ \ct \times \cg } \leq n_\cg < \infty$.
\end{corollaryx}

\subsection{Open questions}
We elaborate on two open questions regarding possible extensions of some of our results.
The first one concerns \cref{corollary:R-times-lattice}; our techniques are not enough to answer this question in the positive since they are restricted to grid posets.

\begin{question}
    Let $\cq$ be a finite poset, and let $\ct$ be an infinite total order.
    Does $\rep(\ct \times \cq)$ admit finite spread-resolutions?
    In that case, is $\gldim{\spread}{ \ct \times \cq }$ finite?
\end{question}

The second question concerns extending \cref{theorem:answers-1-2-lattice} to tame representations in the sense of Miller \cite{miller2}.
An \emph{m-tame}\footnote{We use m-tame rather than ``tame'' to avoid confusion with other notions of tameness.} representation of $\R^n$ is a (not necessarily finitely presented) representation $M : \R^n \to \kvec$ that admits a finite upset presentation (that is, such that $M \cong \coker(Q \to P)$, with $P$ and $Q$ finitely spread-decomposable, with all spreads being upsets); see \cite[Thm.~6.12]{miller2}.
The notion of an m-tame representation is strictly more general than that of a finitely presented representation; for example, the sublevel set persistent homology of two-dimensional Morse functions is typically not finitely presented, but it is m-tame \cite[Cor.~2.11]{budney-kaczynski}, and the same is true for subanalytically constructible sheaves \cite[Thm.~4.5']{miller}.

%\luis{tame representations need assumptions to form an Abelian category, cite Waas}

\begin{question}
    Does there exist a suitable category of m-tame representations of $\R^n$, strictly containing the finitely presented representations, with the property that every object admits a finite spread-resolution?
\end{question}

In the question, we ask for a suitable category of m-tame representations since it is known that the full category of m-tame representations has some pathologies, such as not being Abelian; see \cite[Sec.~4.5]{miller2} and \cite{waas} for the study of candidate subcategories.

%From \cref{theorem:A}, the monotonicity of the spread-global dimension, and results in \cite{BBH2023}, we deduce the following:

%\begin{corollaryx}
%    \label{corollary:finite-poset}
%    Let $\cp$ be a finite poset.
%    There exists a constant $n_\cp \in \mathbb{N}$ such that, for every finite linear order $\mathcal{L}$, we have $\gldim{\mathrm{spread}}{\rep(\mathcal{L} \times \cp )} \leq n_{\cp}$.
%\end{corollaryx}

%We also deduce the following:
%
%\begin{corollaryx}
%    \label{corollary:continuous-poset}
%    If $\ct_1, \dots, \ct_n$ are finite linear orders,
%    then the spread exact structure on $\rep(\mathbb{R} \times \ct_1 \times \cdots \times \ct_n)$ exists and has finite global dimension.
%\end{corollaryx}
%
%To the best of our knowledge, \cref{corollary:continuous-poset} is the first result establishing the finiteness of the spread-global dimension on an infinite, non-linear poset.
%In the result, the subscript $fp$ indicates that we are only considering finitely presented representations.
%%In the same way that \cref{conjecture:1} is relevant to the study of morphisms between representations of finite linear orders, \cref{corollary:spread-global-dimension-product-R} 
%We conjecture that \cref{corollary:R-times-lattice} can be extended as follows:

\subsection{Overview of the approach and structure of the paper}
\label{section:overview-of-approach}
In \cref{section:background} we provide background about
posets (\cref{section:posets}),
poset representations (\cref{section:poset-representations}),
spreads and spread representations (\cref{section:spreads}),
relative homological algebra (\cref{section:relative-homological-algebra}),
irreducible morphisms (\cref{section:irreducible-morphisms-radical}),
and projectivization (\cref{section:projectivization}).
%and
%radical approximations (\cref{section:radical-approximations}).

In \cref{section floor functor}, we prove \cref{theorem:answers-1-2-lattice}, which follows from \cref{theorem:main-extension-theorem}.
%For this, we study left Kan extensions along upper semilattice morphisms, and their action on spreads representations.
The results in this section can be understood as extensions of results that are known for particular choices of spreads (namely principal upsets in the case of the standard exact structure, and differences between principal upsets in the case of the rank exact structure).

In \cref{section:bounding-relative-global-dimension}, we prove \cref{thm: big small radapp}, a main technical contribution, which gives a general approach for bounding the $\add \cx$-global dimension of an algebra $\Lambda$ from above by the $\add \cy$-global dimension of another algebra $\Delta$, in the case where $\add\cx$ contains the image of $\cy$ under a well-behaved family of functors $\fmod \Delta \to \fmod \Lambda$.
In order to prove \cref{thm: big small radapp}, we use the notion of radical approximation, and in particular we rely on \cite{BBH2023} to give an explicit description of spread-radical approximations (\cref{prop: spread radical approx domain}).

In \cref{section:stabilization-dimension}, we prove \cref{theorem:stabilization} and its consequences.
The approach for proving \cref{theorem:stabilization} is to
apply \cref{thm: big small radapp} to the case where 
$\Lambda = \K\cp$ and $\Delta = \K\cq$ are poset algebras for $\cp$ and $\cq$ certain finite grids, and where $\cx$ and $\cy$ are the sets of spread representations of $\cp$ and $\cq$, respectively.
In order to apply \cref{thm: big small radapp}, we need a family of functors $\rep \cq \to \rep \cp$ satisfying two main properties: the functors must preserve spread-resolutions, and every spread-radical approximation over $\cp$ is the image of a spread-radical approximation over $\cq$ along one of these functors.
As family of functors we take all left Kan extensions along aligned grid inclusions $\cq \to \cp$, which satisfy the first property by \cref{theorem:main-extension-theorem}.
The rest of the results in the section concern the second property.

In \cref{section:technical-results-spreads}, we prove useful, elementary results about spreads.
%, which, to our knowledge, are not in the literature.

\subsection{Related work and other remarks}
We start by commenting on relations to \cite{BBH2023} and \cite{AT2025}.
The first-mentioned paper leverages the concept of aligned grid inclusion (\cref{section:aligned-grid-inclusion}) to study exact structures on categories of representations of infinite posets.
The second-mentioned paper (made public after the first version of this paper was made public) generalizes \cref{prop: fus spread exact} as \cite[Thm.~4.10]{AT2025}, among other things.

Note that, in the following remarks, we use notation from the body of this paper.

\begin{rem}
In \cite[Sec.~7.2]{BBH2023}, the contraction functor $\Lan_{\floor{-}_\psi}$ was computed in the case where $\psi$ is an aligned grid inclusion using the colimit
\[
        \left(\Lan_{\floor{-}_\psi} M\right)(q)
        \;=\;
        \colim_{\substack{
                p \in \cp \\
                {\floor p} = q}}\,
            M(p).
    \]
By the uniqueness of adjoints, it follows that this colimit coincides with the one used in the proof of \cref{lem:contraction is left adjoint}, in this case.
This can also be shown directly using \cite[Lem.~7.6]{BBH2023}; see \cite[Prop.~3.5]{AT2025} for a detailed proof.
\end{rem}

\begin{rem}
We briefly explain how our approach to \cref{theorem:answers-1-2-lattice} is related to the concepts developed in \cite[Sec.~7.3]{BBH2023} and \cite[Sec.~4]{AT2025}. Note in particular that,
although related, \cref{theorem:main-extension-theorem} is not a special case of \cite[Thm.~4.10]{AT2025}, since an aligned grid inclusion $\cq \rightarrow \cp$ is only an aligned interior system (in the sense of \cite{AT2025}) if $\cp = \us \cq$.

Let $\cp$ be a grid poset, and let $\mathfrak{H}$ denote the set of finite aligned subgrids of $\cp$. Let $\cx$ denote the set of spread representations of $\cp$, and for $\cq \in \mathfrak{H}$ let $\cx_\cq$ denote the set of spread representations of $\cq$.
In \cite[Def.~7.21]{BBH2023}, the notion of an $(\mathfrak{H},\cp)$-extended projective class was introduced. One can use \cite[Prop.~4.4]{AT2025} to conclude that
$(\cx,\{\cx_\cq\}_{\cq\in\mathfrak{H}})$
satisfies conditions (3) and (4) of this definition. Moreover, conditions (1), (2), and (6) of the definition all hold by standard arguments; see e.g., \cite[Thm~2.10~and~Lem.~7.13]{BBH2023}.
On the other hand,
$(\cx,\{\cx_\cq\}_{\cq\in\mathfrak{H}})$
does not satisfy condition (5) of \cite[Def.~7.21]{BBH2023}; see \cite[Ex.~7.26]{BBH2023}.
To get around this, we recall from (the dual of) \cite[Lem.~4.4]{AENY23} that any quotient of a spread representation is spread-decomposable (\cite{AENY23} shows this only for finite posets, but the result extends to infinite posets).
Thus
$(\cx,\{\cx_\cq\}_{\cq\in\mathfrak{H}})$
satisfies the following weakening of \cite[Def.~7.21(5)]{BBH2023}.

\begin{itemize}
 \item[(5')] Let $\iota:\cq \rightarrow \cp$ be an aligned grid inclusion. Let $N \in \rep\cq$, $X \in \cx$, and $f \in \Hom_\cp(X, \Lan_{\floor{-}_\iota}(N))$.  Then there exists $Z \in \add\cx$ with $\supp(Z) \subseteq \us{(\iota\cq)}$ and a morphism $g \in \Hom_\cp(Z, \Lan_{\floor{-}_\iota}(N))$ such that $f$ factors through $g$.
\end{itemize}
Indeed, one can take $Z$ to be the image of $f$ to satisfy the condition (5').
Using (5') in place of \cite[Def.~7.21(5)]{BBH2023}, one can then conclude \cref{theorem:answers-1-2-lattice} using an argument similar to the proof of \cite[Thm.~7.22]{BBH2023}.
\end{rem}

\begin{rem}
    \label{remark:mistake}
    The first public draft of this paper incorrectly stated \cref{theorem:answers-1-2-lattice} (and \cref{theorem:main-extension-theorem}) in the generality of upper semilattices (rather than grid posets).
    Toshitaka Aoki kindly pointed out the issue to us, and provided the following example, which exhibits a finitely presented representation of an upper semilattice that does not admit a finitely presented spread-approximation. This example shows that \cref{theorem:answers-1-2-lattice} does not generalize from grid posets to upper semilattices.
    The restriction to grid posets is mild, since the infinite posets of interest in persistence theory are typically $\R^n$ and $\Z^n$ \cite{botnan-lesnick}, both of which are grid posets.
    In the example, the bottom row represents a subposet isomorphic to $\{-\infty\} \cup (- \mathbb{N}) = \{-\infty < \cdots < -1 < 0 \}$, where the representation is constantly $\K$ with identity morphisms.
\[
\begin{tikzpicture}
    \matrix (m) [matrix of math nodes,row sep=1.5em,column sep=2em,minimum width=2em,nodes={text height=1.75ex,text depth=0.25ex}]
    {
        \K   &        &    &      \\
        \K^2 &        &    &  \K  \\
        \K   & \cdots & \K &  \K \\};
    \path[line width=0.75pt, ->]
    (m-1-1) edge [right] node {\scriptsize $[0,1]^T$} (m-2-1)
    (m-3-1) edge [right] node {\scriptsize $[1,0]^T$} (m-2-1)
    (m-3-1) edge [-,above] (m-3-2)
    (m-3-2) edge (m-3-3)
    (m-3-3) edge [above] node {\scriptsize $1$}(m-3-4)
    (m-3-4) edge [left] node {\scriptsize $1$} (m-2-4)
    (m-2-1) edge [above] node {\scriptsize $[1,1]$} (m-2-4)
    %(m-1-1) edge [bend left=36, above] node {$\;\;\;\;\Lan_{\floor{-}_\psi}$} (m-1-2)
    %(m-1-2) edge [bend left=-12, above] node {$\Lan_\psi$} (m-1-1)
    %(m-1-1) edge [bend left=-12, above] node {$\res_\psi$} (m-1-2)
    %(m-1-2) edge [bend left=36, above] node {$\Ran_\psi$} (m-1-1)
    ;
\end{tikzpicture}
\]
\end{rem}

%\luis{fix references once the sections are finished}
%In \cref{section floor functor}, 
%we study, for an upper semilattice morphism $\psi : \cq \hookrightarrow \cp$ with $\cq$ finite,
%the left Kan extension $\Lan_\psi : \rep \cq \to \rep \cp$.
%%Stretching functors are standard in persistence theory and more generally the representation theory of posets, and can be defined abstractly as a left Kan extension along a poset morphism (see \cref{section:general-setup-stretching}).
%Our main technical contribution here is \cref{prop: fus spread exact}, which implies that the $\Lan_\psi$ preserves spread-exact resolutions, addressing (1).

%In \cref{section:spread-radical-approximations}, we first
%prove \cref{prop: spread radical approx domain}, which builds on the description of spread-irreducible morphisms of \cite{BBH2023}, and gives an explicit description of spread-radical approximations.
%Then, combining this description with combinatorial arguments on posets, we prove \cref{prop rad approx hit by floor}, which gives sufficient conditions for a spread-radical approximation over $\cp$ to be in the image of $\Lan_\psi$ for some $\psi$, addressing (2).
%
%\smallskip
%
%\cref{theorem:stabilization} is then a consequence of \cref{thm: big small radapp,prop: fus spread exact,prop rad approx hit by floor}.
%\cref{theorem:stheorem:stabilizationcorollary:continuous-poset} are proven in \cref{section:hitting-radical-approximations}.

\subsection*{Acknowledgements}
We thank Thomas Br\"{u}stle and Steve Oudot for various inspiring conversations on these and related topics.
We thank Toshitaka Aoki for pointing out an issue in the first version of this paper (\cref{remark:mistake}).
BB thanks Emerson Escolar, Shunsuke Tada, and Toshitaka Aoki for conversations during a visit to Kobe University, supported in part by JSPS (International Fellowship for Research in Japan) and Mitacs Globalink.
JD was supported by the bourse d'excellence of the Institut des Sciences Mathématiques and by the Université de Sherbrooke under the supervision of Thomas Br\"{u}stle.
LS was supported by a Centre de Recherches Mathématiques-Institut des Sciences Mathématiques fellowship.
EJH was partially supported by Canada Research Chair CRC-2021-00120, NSERC Discovery Grants RGPIN-2022-03960 and RG-PIN/04465-2019, and an AMS-Simons travel grant. A portion of this work was completed while EJH was a postdoctoral fellow at the Université du Québec à Montréal and the Université de Sherbrooke.

%%%%%
\section{Background}
\label{section:background}

We assume familiarity with category theory \cite{maclane,riehl} and the fundamentals of representation theory of finite-dimensional algebras \cite{ARS,assem}.
% and quivers \cite{schiffler}.
In this section we recall the necessary definitions, results, and notations that we will need in the paper.

\subsection{Basic algebraic conventions}
Recall that we fix a field $\K$ throughout the paper.
We let $\kvec \subseteq \kVec$ denote the categories of finite-dimensional and all vector spaces, respectively.
All algebras are assumed to be unital and associative over $\K$.
If $\blam$ is a finite-dimensional algebra, we let $\fmod{\blam}$ denote the category of finitely generated (equivalently, finite-dimensional) right $\blam$-modules.

A \textit{linear Krull--Schmidt category} is an additive $\K$-category $\ca$ such that every object $A \in \ca$ admits a finite direct sum decomposition $A \cong \bigoplus_{i=1}^m A_i$, with $\End_\cc(A_i)$ a local $\K$-algebra for every $i$.
The objects with local endomorphism algebra are called \textit{indecomposable}.
For every finite-dimensional algebra $\Lambda$, the category $\fmod \Lambda$ is a linear Krull--Schmidt category.

Recall that a morphism $f : B \to C$ in a category is \emph{right minimal} if every morphism $g : B \to B$ such that $f \circ g = f$ is an isomorphism.

\subsection{Posets and spreads}
\label{section:posets}

For examples of the notions recalled here, see \cref{figure:cover-cocover-incomparable}.

\noindent Let $\cp$ be a poset.
We recall several definitions associated to a subset $S\subseteq \cp$, viewed as aposet with the same relation
:
\begin{itemize}

    \item Define $\us{S} \coloneqq \{ p\in \cp : \exists s\in S, s\leq p\}$ and
    $\ds{S} \coloneqq \{ p\in \cp : \exists s\in S, s\geq p\}$;

    \item The set $S$ is an \textit{upset} (resp.~\textit{downset}) if $S = \us{S}$ (resp.~$S = \ds{S}$),
    and a \emph{principal upset}
    (resp.~\textit{principal downset}) if $S = \us{\{p\}}$
    (resp.~$S = \ds{\{p\}}$), for some $p \in \cp$;

    \item The set $S$ is an \textit{antichain} if any two distinct elements $s \neq s'\in S$ are incomparable.

    \item The subsets $\minimal S \subseteq S$ and $\maximal S \subseteq S$ are the sets of \emph{minimal} and \emph{maximal} elements of $S$, respectively.

    \item
    The set of \emph{covers} of $S$
    is $\cover(S) \coloneqq \min(\us{S} \setminus S)$.

    \item
    The set of \emph{cocovers} of $S$
    is $\cocover(S) \coloneqq \max(\ds{S} \setminus S)$.

\end{itemize}
If $A,B \subseteq \cp$, define $\iprod{A}{B} \coloneqq (\us A) \cap (\ds B) =\{ p\in \cp : \exists a\in A, \exists b\in B, a\leq p\leq b\}$.
\begin{itemize}

    \item The set $S$ is \textit{convex} if $\iprod{\{s\}}{\{s'\}} \subseteq S$ for all $s,s' \in S$;
    
    \item The set $S$ is \textit{connected} if $S\neq \emptyset$ and for every $s,s' \in S$, there is a \textit{zigzag path} between from $s$ to $s'$ contained in $S$, that is, a sequence $s = s_0 \leq s_1 \geq s_2 \leq \cdots \geq s_k = s' \in S$;

    \item The set $S$ is a \textit{spread} if it is convex and connected;

    \item The set $S$ is a \textit{connected component} if it is connected and not properly contained in another connected subset.
\end{itemize}
For $p \in \cp$, we often denote the singleton set $\{p\}$ by $p$.

\begin{eg}
    In the poset of real numbers $\R$, a closed interval $[a,b] = \iprod{a}{b} \subseteq \R$ has no cocovers or covers, while a half-open interval $[a,b) = \us a \setminus \us b\subseteq \R$ has $b$ as its only cover.
\end{eg}

\begin{figure}
    \includegraphics[width=0.6\linewidth]{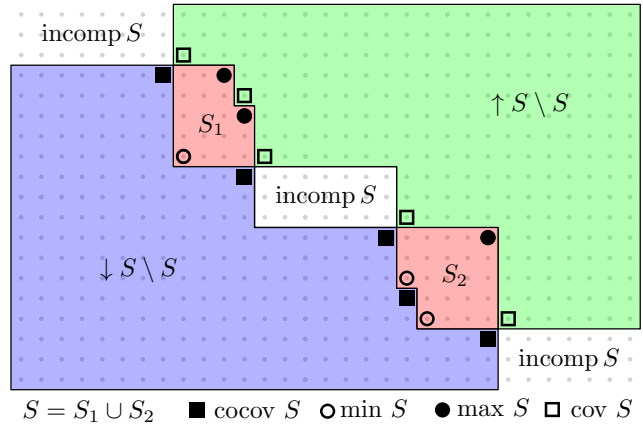}
    \caption{Dots represent the finite grid poset $\cp = [31] \times [19]$.
        Note that in all figures in this paper, the poset order of a two-dimensional grid poset increases from left to right, and from bottom to top.
        The subsets $S_1,S_2 \subseteq \cp$ are spreads, and the union $S = S_1 \cup S_2$ is convex but not connected.
        Illustrated are also the minimal elements, maximal elements, covers, and cocovers of $S$, as well as $\us{S} \setminus S$ and $\ds{S} \setminus S$. \cref{cor: cover iff min upminus} describes the set $\incomp S$ of elements which are not comparable to any element of~$S$.
        }
    %\caption{ An example of \cref{cor: cover iff min upminus} for the convex set $S$, depicted in red. The complement of $S ~\sqcup \us{\cover{ S}} ~\sqcup \ds{\cocover S}$, depicted in white, are exactly the elements which are not comparable to any element of $S$.Note that the covers of $S$ are the minima of $\us S\setminus S$ as in \cref{lemma:description-finitely-presented-spreads}.}
\label{figure:cover-cocover-incomparable}
\end{figure}

%\luis{Here, I have commented out a lemma with several easy properties of the definitions of this subsection.
%If any of these properties are used, they should just be proved where they are used.
%The lemma is still commented, for reference.}\justin{I believe that all the uses of these lemmas have been accounted for (may 28)}
\begin{comment}
The following standard results are straightforward to verify.

\todo{cite lemmas where they are used. delete afterwards}
\begin{lem}
    Let $\cp$ be a poset.
    \begin{enumerate}
    \item The intersection of convex subsets of $\cp$ is convex.
    \item Each connected component of a convex subset of $\cp$ set is a spread.
    \item If $S \subseteq \cp$, the upset $\us S $ and downset $\ds S $ are convex. 
    \item If $S \subseteq \cp$, the sets $\minimal S$ and $\maximal S$ are antichains.
    \item If $S \subseteq \cp$ is finite, then $\us S = \us \minimal S$, $\ds S = \ds \maximal S$, and $(\us S) \cap (\ds S) = \iprod{\minimal S}{\maximal S}$.
    \item
    For any $A, B \subseteq \cp$, set $\usminus{A}{B}$ is convex.
    \item
    \todo{This statement seems wrong! Is it used somewhere?}
    If $A, B \subseteq \cp$ are such that $\cp$ contains a least upper bound for $A \cup B$, then $\usminus{A}{B}$ is a spread.
    \end{enumerate}
\end{lem}
\end{comment}

\subsection{Posets of interest}
\label{section:posets-of-interest}

An \emph{upper semilattice} is a poset $\cp$ where every finite, non-empty set of elements $S \subseteq \cp$ admits a join (i.e.~least upper bound), denoted $S$.
%If the empty set also admits a join (equivalently, $\cp$ has a least element), then we say that $\cp$ is a \emph{bounded upper semilattice}.

A \textit{grid poset} is a poset of the form $\cg = \ct_1 \times \cdots\times \ct_n$, where each $\ct_i$ is a totally ordered set, and the order on $\cg$ is the product order, that is
\[
    (x_1,\dots , x_n) \leq (y_1,\dots, y_n)
    \text{ if and only if }
    x_i \leq y_i \in \ct_i \text{ for all } 1\leq i \leq n. 
\]
We denote by $\pi_i : \cg \to \ct_i$ the projection map $(x_1,\dots , x_n) \mapsto x_i$.
Note that every grid poset is an upper semilattice.

If $\cg$ is a grid poset and $\ct$ is a totally ordered set, then $\ct \times \cg$ is again a grid poset.
In this case, the projection denoted by $\pi_0$ represents the projection onto $\ct$.

\subsection{Poset representations}
\label{section:poset-representations}
Let $\cp$ be a poset.
The category associated to $\cp$, also denoted by $\cp$, has as objects the elements of $\cp$, precisely one morphism $x \to y$ when $x \leq y \in \cp$, and no morphisms $x \to y$ otherwise.
A \emph{representation}
    %\luis{Also note that in older references such as \cite{nazarova-roiter} poset representations mean a different thing?}}
of $\cp$ is a functor $\cp \to \kVec$.
Note that, in the persistence literature, representations of posets are also known as persistence modules.

Poset representations of a finite poset $\cp$ can be seen as modules
over the incidence algebra of $\cp$.
Recall that the \emph{incidence algebra} of $\cp$, denoted $\K\cp$, has as underlying vector space the free vector space generated by pairs $[i,j]$ with $i \leq j \in \cp$, and multiplication given by linearly extending the rule $[i,j][k,l] = [i,l]$ if $j = k$, and $[i,j][k,l] = 0$ otherwise. 

\medskip

\noindent
\begin{center}
    \fbox{%
        \begin{minipage}{0.98\linewidth}
            \begin{rem}
                \label{remark:poset-representations-as-modules}
                When $\cp$ is a finite poset, the algebra~$\K\cp$ is unital and finite dimensional, and there is an equivalence of categories $\rep \cp \simeq \fmod{\K\cp}$ (see, e.g., \cite[Lem.~2.1]{botnan-oppermann-oudot-scoccola}).
                By a standard abuse of notation, we identify these two categories.
            \end{rem}
        \end{minipage}}
\end{center}

\begin{rem}
    \label{remark:poset-duality}
    The \emph{opposite} of a poset $\cp$, denoted $\cp^\op$, has the same underlying set as $\cp$ and the order relation $p \leq p' \in \cp^\op$ if and only if $p' \leq p \in \cp$.
    When $\cp$ is viewed as a category, the opposite $\cp^\op$ is usual opposite category \cite[Ch.~1,~Def.~2.1]{riehl}.
    %We can therefore consider the notion of duality in a poset. 
    Any statement that applies to all posets thus has a dual statement, obtained by replacing all posets by their duals, and the original statement is true if and only if its dual is. 

    In particular, a subset $S\subseteq \cp$ of a poset is convex (resp.~connected, an antichain) if and only if the subset $S^\op \subseteq \cp^\op$ is convex (resp.~connected, an antichain) in the opposite poset.
    Similarly, we have $\max S = \min(S^\op)$, $\cover S = \cocover(S^\op)$, and $\us{S} = \ds{(S^\op)}$.
    %These results can be observed by looking at \cref{figure:cover-cocover-incomparable} upside-down.
% It is easy to see that maxima are dual to minima, covers are dual to cocovers, and upsets are dual to downsets. 
%Furthermore, a subset $S\subseteq \cp$ is convex (\emph{resp.} connected) if and only
\end{rem}

\subsection{Spread representations}
\label{section:spreads}

If $S \subseteq \cp$ is convex, define the \textit{indicator representation} $\K_S \in  \rep\cp$ by
\begin{align*}
    \K_S(p) = \begin{cases}
    \K & p\in S \\ 0 & \text{otherwise}
\end{cases} 
\hspace{.5cm} \text{and} \hspace{.5cm} \K_S(p\leq p') = \begin{cases}
    \id_\K & p,p' \in S \\ 0 & \text{otherwise}
\end{cases} 
\end{align*}
for all $p,p' \in \cp$.
If $S$ is a spread, the representation $\K_S$ is called a \emph{spread representation}.

%By a standard abuse of language, we say that a representation of $\cp$ is a \emph{spread representation} if it is isomorphic to $\K_S$ for some spread $S \subseteq \cp$.
Note that all spread representations are indecomposable, since their endomorphism algebra is $\K$.

The following result is classical; see, e.g., \cite[Lem.~2.4]{botnan-oppermann-oudot-scoccola} for a proof.

\begin{lem}
    \label{lemma:proj-is-spread}
    Let $\cp$ be a finite poset.
    A representation $M :\cp \to \kvec$ is projective if and only if it is a finite direct sum of spread representations of the form $\K_{\us{p}}$ for $p \in \cp$.
    Dually, a representation is injective if and only if it is a finite direct sum of spread representations of the form $\K_{\ds{p}}$ for $p \in \cp$.
    Thus, over a finite poset, the additive closure of the spread representations contains all projective and injective representations.
    \qed
\end{lem}

\subsection{Finitely presented representations}

A representation $M$ of $\cp$ is \emph{finitely presented} if it is the cokernel of a morphism $C_1 \to C_0$, where $C_0$ and $C_1$ are finite direct sums of representations of the form $\K_{\us{p}}$, for $p \in \cp$ (i.e. representable functors).
The category of finitely presented representations of $\cp$ is denoted by $\rep \cp$.

A convex subset $S\subseteq \cp$ is \emph{finitely presented} if $\K_\cp$ is finitely presented.
Thus, if $\cp$ is finite, every convex set is finitely presented.

The following are important properties of finitely presented convex sets.
Proofs are in 
\cref{section:technical-results-spreads}.

\begin{lem}
    \label{lemma:description-finitely-presented-spreads}
    Let $\cp$ be a poset.
    If $S \subseteq \cp$ is a finitely presented convex set, then
    $\min(S)$ and $\cover(S)$ are finite and $S = \us{\min(S)} \setminus \us{\cover(S)}$.
\end{lem}

\begin{lem}\label{cor: cover iff min upminus}
    If $S$ is a finitely presented convex subset of a poset $\cp$,
    then $\us S = S \sqcup \us{\cover S}$, $\ds S = S \sqcup \ds{\cocover S}$, and
    $\cp = S \sqcup (\us{\cover{ S}}) \sqcup (\ds{\cocover S}) \sqcup (\incomp S)$,
    where $\incomp S = \cp \setminus(\ds S \cup \us S ) $ consists of all elements that are not comparable to any element of $S$.
\end{lem}

%\begin{lem}
%    \label{lemma:restriction-has-adjoints}
%    Let $f : \cq \to \cp$ be a poset map such that \luis{I still need to see what is a good set of assumptions}.
%    The restriction functor $f^* : \rep \cp \to \rep \cq$ given by precomposition with $f$ admits both a left and a right adjoint.
%\end{lem}
%\begin{proof}
%    \todo{}
%\end{proof}

\subsection{Poset morphisms}
\label{section:poset-morphisms}
A \emph{poset morphism} is a function $\psi : \cq \to \cp$ between posets such that $q \leq q' \in \cq$ implies $\psi(q) \leq \psi(q')$.
Such poset morphism is \emph{full} if $\psi(q) \leq \psi(q') \in \cp$ implies $q \leq q' \in \cq$ for every $q,q'\in\cq$.
Note that a full poset morphism is in particular injective.
An \emph{upper semilattice morphism} is a morphism between upper semilattices that preserves all finite, non-empty joins.

In the next subsection we introduce a particular kind of upper semilattice morphism that is central to this paper.

\subsection{Aligned grid inclusions}
\label{section:aligned-grid-inclusion}

Let $\cq_j$ and $\cp_j$ be totally ordered sets for $1 \leq j \leq n$, and let $\cq = \cq_1 \times \cdots \times \cq_n$ and $\cp = \cp_1 \times \cdots \times \cp_n$ be the corresponding grid posets.
An \emph{aligned grid inclusion} from $\cq$ to $\cp$ is a poset morphism $\iota : \cq \to \cp$
that is equal to a product $\iota_1 \times \cdots \times \iota_n$ of injective poset morphisms $\iota_j : \cq_j \hookrightarrow \cp_j$.

The following is straightforward to check.

\begin{lem}
    \label{lemma:aligned-grid-inclusion-semilattice-morphism}
    Every aligned grid inclusion is a full upper semilattice morphism.
    \qed
\end{lem}

\begin{eg}
    An aligned grid inclusion into a product poset $\{0 < \cdots < a\} \times \{0 < \cdots < b\}$ can be obtained by deleting some columns, as in \cref{figure:example-aligned-grid-inclusion-and-floor}.
\end{eg}

If $\cp$ and $\cq$ have a minimum element, an \emph{origin aligned grid inclusion} from $\cq$ to $\cp$ is an aligned grid inclusion that maps the minimum of $\cq$ to the minimum of $\cp$.

\subsection{Homological algebra relative to a subcategory}
\label{section:relative-homological-algebra}
%\begin{defn}\label{def: homological}
Let $\cc \subseteq \ca$ be a full subcategory, with $\ca$ a linear Krull--Schmidt category, and let $M \in \ca$.
\begin{itemize}
    \item A (minimal) \textit{$\cc$-approximation} is a (right minimal) morphism $f:C \to M$ such that $C\in \cc$ and such that $\Hom_\ca(Z, f) : \Hom_\ca(Z,C) \to \Hom_\ca(Z,M)$ is surjective for all $Z\in\cc$.
    
    \item The sequence $C_\bullet$ in the top row of the following diagram is 
    a \textit{(minimal) $\cc$-resolution} of $M$ if each $g_i: C_i \to \ker f_{i-1}$ is a (right minimal) $\cc$-approximation:
    \begin{equation*}
        % https://tikzcd.yichuanshen.de/#N4Igdg9gJgpgziAXAbVABwnAlgFyxMJZABgBpiBdUkANwEMAbAVxiRAB12BjKCHBAL6l0mXPkIoAjOSq1GLNgE0A+gCYQQkdjwEiqmdXrNWiECskbhIDNvFEAzAbnGly4pa1jdKACxOjCqYAsh7WojoSyNKSsgEmHOwA1jAATgAEAGZqoTZekfoxhvLxnMnpWRaaYbbeyI6FzoEJZZluOeF2vqQNcWylqa3AALSSAu01kQCs-sVs7gKyMFAA5vBEoBkpEAC2SGQgOBBI0o3xWfYg1Ax0AEYwDAAKHd4gDDAZOKGbO3vUh0jTU5sZbZK63e5PCZsN4fL5bXaIE7-RD6IGmLLqMF3R7PCSvd6fKrfBFIo6IABsRRcphBmNe4JxUNMMMJVmJSFRyMcaJAFUu9OxkLy0IJcJ+KL+ZIA7FSmiCLFiIbiRbCifCkNzkX4eVl3IrGcLmaK1eLNWSAByy+IgvUCpVM-GqtnqxCA5EnXqIMBMBgMfVCiIqz5-OhYBhsAAWEAgiTFCMpBzJqM93t9-uVRthIbDkejsZNCJliY1VqQqb9doNgczwYOofDpijMbjSEtxcQ2pTPor10FGcdtZw9dzzYLSG1yIAnKX0cphqMNBQBEA
        \begin{tikzcd}
        \cdots \arrow[r, "f_3"] \arrow[rd, "g_2"] & C_2 \arrow[r, "f_2"] \arrow[rd, "g_2"] & C_1 \arrow[r, "f_1"] \arrow[rd, "g_1"] & C_0 \arrow[r, "f_0"] \arrow[rd, "g_0"] & M \arrow[r, "f_{-1}"]       & 0 \\
                                                  & \ker f_2 \arrow[u, hook]               & \ker f_1 \arrow[u, hook]               & \ker f_0 \arrow[u, hook]               & \ker f_{-1} \arrow[u, hook] &  
        \end{tikzcd}
        \label{eq:def resolution}
    \end{equation*}
    In that case, the subsequence $C_1 \xrightarrow{f_0} C_1 \xrightarrow{f_0} M$ is a \textit{(minimal) $\cc$-presentation} of $M$.

    \item The \textit{length} of a $\cc$-resolution $C_\bullet \to M$ is $\min (\{i\in \N : C_{i+1}=0 \} \cup \{\infty\}  )$.

    \item If $M$ admits a $\cc$-resolution, the \textit{$\cc$-dimension} of $M$, denoted $\pdim{\cc}{M}$, is the smallest length of a $\cc$-resolution of~$M$.

    \item If every object of $\ca$ admits a $\cc$-resolution, the \textit{$\cc$-global dimension} of $\ca$ is $\gldim{\cc}{\ca} \coloneqq \sup \{ \pdim{\cc}{N} : N \in \ca\}$.
\end{itemize}
%\end{defn}

Note that a $\cc$-approximation is also known as a right $\cc$-approximation or a $\cc$-precover.

%Although we will not use this, we note that, in the above setting, the following class of morphisms forms an exact structure on $\ca$ in the sense of \cite[Prop. 1.7]{Draxler_exact_cats}:
%\luis{I think the following is not right, the sequences should be those orthogonal to the objects in $\cx$}
%\[
%    \left\{\, \ker \into C \xrightarrow{f} M : f \text{ is a right $\cc$-approximation}\,\right\}.
%\]
%\justin{since this is false, we should just cut it. we don't need it}
\begin{eg}
    Let $\bgam$ be a $\K$-algebra and consider the subcategory $\proj \bgam \subseteq \fmod \bgam$ of projective modules.
    Then, the notions of $\proj \bgam$-approximation, $\proj \bgam$-resolution, $\proj \bgam$-exactness,
    $\proj\bgam$-dimension, and $\proj\bgam$-global dimension coincide with the usual notions of $\bgam$-projective precover, $\bgam$-projective resolution, $\bgam$-projective dimension, and global dimension of $\bgam$, respectively.
\end{eg}

Often, the subcategory $\cc\subseteq \ca$ will be the \emph{additive closure} $\add\cx$ of a collection of objects $\cx \subseteq \ca$, that is, the full additive subcategory of $\ca$ consisting of all objects that are isomorphic to a direct summand of a direct sum of objects in $\cx$.

\smallskip

\noindent
\begin{center}
    \fbox{%
        \begin{minipage}{0.98\linewidth}
\begin{notation}
    If $\cp$ is a poset, we let $\addspread\cp$ denote the additive closure of the set of finitely presented spread representations.
    When the poset $\cp$ is clear from the context, we write ``spread'' instead of $\addspread\cp$.
    This means that we write spread-approximation, spread-resolution, spread-exact, spread-dimension, and spread-global
    for the corresponding notions relative to $\addspread\cp$.
\end{notation}
\end{minipage}}
\end{center}

We conclude with a remark connecting relative homological algebra as described in this section to exact structures of the form $\cf_\cx$ as described in the introduction. See \cite{draxler-reiten-smalo-solberg} for further details.

\begin{rem}
    \label{remark:relative-precov-surjective}
    Let $\cx$ be a class of objects of $\ca$ such that every object of $\ca$ admits a finite $\add\cx$-resolution.
    Assume, moreover, that $\proj\ca \subseteq \add \cx$.
    Then, the following are true and easy to check:
    \begin{itemize}
    \item The projective objects of the exact structure $\cf_\cx$ are precisely the objects in $\add\cx$.
    \item Projective resolutions relative to $\cf_\cx$ are the same as the $\add\cx$-resolutions.
    \item The global dimension of $\cf_\cx$ is the same as the $\add\cx$-global dimension.
    \end{itemize}
    %Suppose $\ca$ is an Abelian category with enough projectives.
    %If $\cc \subseteq \ca$ is a full subcategory containing $\proj\ca$, then $\cc$-approximations are epimorphic, and $\cc$-resolutions are exact sequences.
    This is one reason why several references on relative homological algebra assume $\proj\ca \subseteq \add\cx$.
    This assumption is however not necessary for
    some of the results in \cref{section:bounding-relative-global-dimension}.
\end{rem}

\subsection{Irreducible morphisms and the radical}
\label{section:irreducible-morphisms-radical}

A morphism $f:M \to N$ of a category $\cc$ is \textit{$\cc$-irreducible} if it is neither a section nor a retraction, and is such that for any factorization $f = hg$ in $\cc$, either $g$ is a section or $h$ is a retraction.
The reason why we say ``$\cc$-irreducible'' and not just ``irreducible'' is that, in our applications, the category $\cc$ is a full subcategory of a larger category $\ca$, as in \cref{section:relative-homological-algebra}.

\begin{eg}
    \label{example:irreducible-proj-poset-case}
    Let $\cp$ be a finite poset and consider the full subcategory $\proj\cp \subseteq \rep\cp$ of projective representations, so that
    $\proj \cp = \add{\{ \K_{\uparrow i} \}_{i \in \cp}}$, by \cref{lemma:proj-is-spread}.
    If $i \leq j \in \cp$, then there is a non-zero morphism $\K_{\us j} \to \K_{\us i}$, and this morphism is $\proj\cp$-irreducible if and only if $j$ is a cover of $i$.
\end{eg}
\begin{defn}\label{def: cat rad}
    The \emph{radical} of a linear Krull--Schmidt category $\cc$
    is the additive subfunctor $\rad_\cc: \cc^{op} \times \cc \to \kvec$ of $\Hom_\cc(-,-)$, characterized by the following properties:
    \begin{enumerate}
        \item If $X ,Y \in \cc$ are indecomposable then
        \[
        \rad_\cc(X,Y) =\{ f\in \Hom_\cc(X,Y) :~ f \text{ is not invertible}\}.
        \]
        \item $\rad_\cc(X \oplus Z, Y) = \rad_\cc(X,Y) \oplus \rad_\cc( Z, Y)$, for any $X,Y,Z \in \cc$.
        \item $\rad_\cc(X, Y \oplus Z) = \rad_\cc(X,Y) \oplus \rad_\cc( X, Z)$, for any $X,Y,Z \in \cc$.
    \end{enumerate}
    Define a subfunctor $\rad_\cc^2(-,-) \subseteq \rad_\cc(-,-)$ by
    \begin{align*}
         \rad_\cc^2(M,N) &= \left\{ M\xrightarrow{g} Z\xrightarrow{h} N  ~:~ g\in \rad_\cc(M,Z) \text{ and }h\in \rad_\cc(Z,N) \right\} 
        \\&= \bigcup_{Z\in\cc}\rad_\cc(M,Z) \circ \rad_\cc(Z,N) ~. 
    \end{align*}
\end{defn} 

The following properties of the radical are standard.

\begin{lem}[{\cite[Ch.~V,~Prop.~7.1]{ARS}}]
\label{lem: cat rad properties}
    Let $\cc$ be a linear Krull--Schmidt category, and let $M,N \in \cc$. 
    \begin{enumerate}
        \item The vector spaces $\rad^2_\cc(M,N) \subseteq \rad_\cc(M,N) \subseteq \Hom_\cc(M,N)$ are right $\End_\cc(M)$-modules. An endomorphism $\gamma:M\to M$ acts on $f:X\to Y$ by pre-composition.
        
        \item The radical $\rad_\cc(M,M)$ is the Jacobson radical of the $\K$-algebra $\End_\cc(M)$. Furthermore, $\rad_\cc(M,N) = \Hom_\cc(M,N)\circ \rad_\cc(M,M)$.

        \item If $M$ is indecomposable, then $\rad_\cc(M,M)$ is the unique maximal ideal of $\End_\cc(M)$.
        Therefore, the quotient ${\Hom_\cc(M,M)}/{\rad_\cc(M,M)}$ is a division ring, and the nonzero elements are in bijection with the automorphisms of $M$.

        %\item The radical $\rad_\cc(M,M)$ annihilates the quotient ${\rad_\cc(M,N)}/{\rad_\cc^2(M,N)}$. 
    \end{enumerate}
\end{lem}

The following lemma is a simple generalization of \cite[Ch.~V,~Prop.~7.3]{ARS};
the proof is analogous, by replacing each instance of $\fmod\bgam$-irreducible with $\cc$-irreducible.
\begin{lem}\label{lem: irr iff rad}
    Let $\cc$ be a linear Krull--Schmidt category, and $X, Y \in \cc$ be indecomposables.
    Then 
    $f \in \Hom_\cc( X , Y)$ is $\cc$-irreducible if and only if 
    $f \in \rad_\cc(X,Y) \setminus \rad_\cc^2(X,Y)$. \qed
\end{lem}

The following result is used in \cref{section:bounding-relative-global-dimension}.

\begin{lem}
    \label{lem: fully faithful preserve minimality}
    Let $\phi:\ca \to\cb$ be a fully faithful functor.
    A morphism $f \in \Hom_\ca(M, N)$ is right minimal if and only if $\phi f \in \Hom_\cb(\phi M, \phi N)$ is right minimal.
    If $\ca$ and $\cb$ are Krull--Schmidt categories, the functor $\phi$ induces an isomorphism $\rad_\ca(M,N) \cong \rad_\cb(\phi M, \phi N)$.
\end{lem}
\begin{proof}
    We start with the first statement.
    Since $\phi$ is fully faithful, the functor $\phi$ induces an isomorphism $\End_\ca(M) \cong \End_\cb(\phi M)$.
    Let $\phi h \in \End_\cb(\phi M)$ and suppose $\phi f= (\phi f) (\phi h) = \phi( f h)$. Since $\phi$ is faithful, this occurs exactly when $f = f h$.
    The morphism $h$ is an automorphism if and only if $\phi h$ is an automorphism, so the result follows. 

    For the second statement, we may assume that $M$ and $N$ are indecomposable, by additivity.
    In this case, we have that $\rad_\ca(M,N)$ and $\rad_\cb(\phi M, \phi N)$ are exactly the non-isomorphisms, and $\phi$ induces an equivalence between these since it is fully faithful.
\end{proof}

\subsection{Projectivization}\label{section:projectivization}
In our cases of interest, we do homological algebra over a category $\ca = \fmod \blam$ relative to a subcategory $\cc = \add{\cx}$, where $\blam$ is a finite-dimensional algebra, and $\cx$ is a finite collection of $\bgam$-modules.

In this setup, which we formalize next, relative homological algebra can be reduced to standard homological algebra in the category of modules over a different finite-dimensional algebra.

\begin{assumptions}
    \label{definition:projectivization-setup}
    The \emph{projectivization setup} is one in which
    we are given a finite-dimensional algebra $\blam$ and a finite set $\cx$ of pairwise non-isomorphic indecomposables of $\fmod \blam$, and define
    \[
        G \coloneqq \bigoplus_{X\in \cx} X\, ,\;\;\;\;\;
        \bgam \coloneqq \End_\blam(G) \, ,\;\;\;\;\;
        H \coloneqq \Hom_\blam(G,-) \,.
    \]
\end{assumptions}

Let us assume that we are in the projectivization setup.
Since $H$ is an additive functor, the algebra $\bgam$ decomposes as a direct sum of $\bgam$-modules $\bgam \cong \bigoplus_{X \in \cx} H X$
and the category $\fmod\bgam$ is an Abelian, linear Krull--Schmidt category.

The following standard result allows us to reduce the homological algebra of $\fmod \blam$ relative to $\add \cx$ to the standard homological algebra of $\fmod \bgam$:

\begin{prop}[{\cite[Ch.~II,~Prop.~2.1]{ARS}}]
\label{add G resolution thm}
    In the projectivization setup (\cref{definition:projectivization-setup}), the following hold:
    \begin{enumerate}
        \item The algebra $\bgam$ is finite dimensional, and, for every $M \in \fmod\blam$, we have that $H M$ is a right $\bgam$-module, where an endomorphism $\gamma\in \bgam$ acts on a morphism $g:G\to M$ by pre-composition.

        \item For any $A \in \add{\cx}$ and any $M \in \fmod \blam$, the functor $H$ induces a linear isomorphism
        $\Hom_\blam( A, M) \cong~ \Hom_\bgam( H A,  H M)$, which is functional in $A$ and in $M$.
        \item The functor $H$ restricts to an equivalence of categories between $\add{\cx}$ and $\proj{\bgam}$.
        In particular, the module $H X$ is an indecomposable projective $\bgam$-module for every $X\in \cx$. 

        \item For any $A \in \add{\cx}$ and any $M \in \fmod \blam$, a morphism $A \to M$ is an $\add\cx$-approximation if and only if $H A \to H M$ is a $\proj\bgam$-approximation.

        \item For every $M \in \fmod \blam$, we have $\pdim{\acx}{M} = \pdim{\proj\bgam}{H M}$.
\end{enumerate}
\end{prop}

An important consequence of \cref{add G resolution thm}, and the existence of projective covers \cite[Ch.~I,~Thm.~4.2]{ARS}, is the existence of minimal $\add\cx$-approximations.

\begin{cor}
    \label{cor:approx exist krull}
    In the projectivization setup (\cref{definition:projectivization-setup}), every $M\in \fmod\blam$ admits a minimal $\add\cx$-approximation, and
    every $\add\cx$-approximation factors through a minimal one.
    \qed
\end{cor}
%\begin{proof}
%    Since $H M$ is a module over the finite dimensional algebra $\bgam =HG$, there exists a surjection $\bgam^m \onto HM$, for some $m\in \N$, and,
%    by \cref{add G resolution thm}(4), the corresponding morphism $G^m \to M$ is an $\add\cx$-approximation.
%    To prove existence of minimal $\add \cx$-approximations, we use \cite[Ch.~I,~Thm.~2.2]{ARS}, which says that for any $f\in \Hom_\blam(M,N)$ there exists a decomposition $M = M_1 \oplus M_2$ such that $f|_{M_1}$ is right minimal and $g|_{M_2} =0$.
%\end{proof}

We conclude with an important definition for bounding relative global dimension.

\begin{defn}
In the projectivization setup (\cref{definition:projectivization-setup}), 
the \emph{relative simple} modules are the simple $\bgam$-modules, that is, the modules $S_X \in \fmod \bgam$, with $X \in \cx$ and $S_X := {HX}/{\rad_\Lambda(G,X)}$.
\end{defn}

\noindent Note that $S_X$ is not necessarily in the image of $H$.

\section{Left Kan extension of spread-approximations}
\label{section floor functor}

In this section, we prove \cref{theorem:main-extension-theorem} and \cref{theorem:answers-1-2-lattice}, in that order.
We start with a discussion of the main constructions involved.

\subsection{The general setup}
\label{section:general-setup-stretching}
Since many similar constructions appear in the literature, let us first recall the main setup and notation.
For the sake of simplicity, we tell the story for finite posets $\cq$ and $\cp$, but several constructions generalize to the case of upper semilattices; see \cref{section:left-kan-extension-section}.
See \cref{figure:example-aligned-grid-inclusion-and-floor}
for an illustration of some of the main definitions that are relevant to this section.

Let $\psi : \cq \to \cp$ be a poset morphism, and define the \emph{restriction functor} $\res_\psi : \rep \cp \to \rep \cq$ to be given by precomposition with~$\psi$, that is, for $M : \cq \to \kvec$ we let $\res_\psi M = M \circ \psi$.
The functor $\res_\psi$ admits a left adjoint $\Lan_\psi$, as well as a right adjoint $\Ran_\psi$, given by left Kan extension and right Kan extension along $\psi$, respectively.
If, furthermore, the poset morphism $\psi$ is part of an adjunction $\psi \dashv \floor{-}_\psi$ (also known as Galois connection),
then
$\Lan_\psi = \res_{\floor{-}_\psi}$ and
$\Lan_\psi$ admits a left adjoint $\Lan_{\floor{-}_\psi}$.
We represent this in the following diagram, where each functor is the left adjoint of the functor below it.
\[
\begin{tikzpicture}
    \matrix (m) [matrix of math nodes,row sep=1.5em,column sep=12em,minimum width=2em,nodes={text height=1.75ex,text depth=0.25ex}]
    {
        \rep \cp      & \rep \cq \\};
    \path[line width=0.75pt, ->]
    (m-1-1) edge [bend left=36, above] node {$\;\;\;\;\Lan_{\floor{-}_\psi}$} (m-1-2)
    (m-1-2) edge [bend left=-12, above] node {$\Lan_\psi$} (m-1-1)
    (m-1-1) edge [bend left=-12, above] node {$\res_\psi$} (m-1-2)
    (m-1-2) edge [bend left=36, above] node {$\Ran_\psi$} (m-1-1)
    ;
\end{tikzpicture}
\]

%We refer to $\Lan_\psi$ as the \emph{stretching functor} associated to $\psi$.
For context, we mention that the functor $\Lan_\psi$ is also known as the induction functor,
and the functor $\Ran_\psi$ is also known as the coinduction functor \cite{simson,aoki-escolar-tada}.

The functor $\Lan_{\floor{-}_\psi} : \rep \cp \to \rep \cq$ is important in our approach, and in \cref{lem:contraction is left adjoint} we give sufficient conditions for its existence.
This functor was already studied in \cite{BBH2023}, where it is called the \emph{contraction functor}.
Here, we prove a new, important property of the contraction functor, namely that it maps spread-decomposable representations to spread-decomposable representations (\cref{prop: fls respect spread}).
Note we do not make use of the functor $\Ran_\psi$, but we mention it for completeness, since it plays an important role in, e.g., \cite{aoki-escolar-tada}.

\begin{figure}
    \includegraphics[width=.9\linewidth]{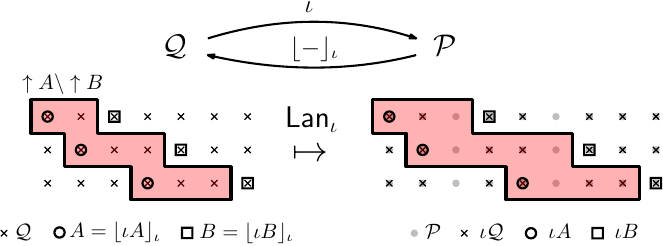}
    \caption{\emph{Left.} A finite grid poset $\cq = [7] \times [3]$ and a spread $\usminus{A}{B} \subseteq \cq$.
    \emph{Right.} A finite grid poset $\cp = [9] \times [3]$.
    The crosses denote the image of an origin aligned grid inclusion $\iota : \cq \to \cp$ (\cref{section:aligned-grid-inclusion}), which is in particular an upper semilattice morphism (\cref{lemma:aligned-grid-inclusion-semilattice-morphism}).
    Note that $\cp = \us{\iota \cq}$, and that $\floor{-}_\iota$ is a retraction for $\iota$, as in \cref{lemma:retraction-for-f}.}
    \label{figure:example-aligned-grid-inclusion-and-floor}
\end{figure}

\subsection{Left Kan extension along an upper semilattice morphism}
\label{section:left-kan-extension-section}

Recall that the definitions of upper semilattice, of full poset morphism, and of upper semilattice morphism are in \cref{section:posets-of-interest} and \cref{section:poset-morphisms}.
%\cref{theorem:answers-1-2-lattice}
%is an easy consequence of \cref{theorem:main-extension-theorem}, so let us describe the structure of the proof of 
%\cref{theorem:main-extension-theorem}.

\begin{lem}
    \label{lemma:retraction-for-f}
    Let $\psi : \cq \to \cp$ be a full, upper semilattice morphism, with $\cq$ finite and $\us{\psi(\cq)} = \cp$.
    The map ${\floor -}_\psi : \cp \to \cq$ given by ${\floor p}_\psi \coloneqq \vee \{q \in \cq : \psi(q) \leq p \}$
    is well-defined and a poset morphism.
    Moreover, the morphism ${\floor -}_\psi$ is the right adjoint of $\psi$, and a retraction for $\psi$.
\end{lem}
%\begin{lem}
%    \label{lemma:retraction-for-f}
%    Let $\psi : \cq \to \cp$ be a full, upper semilattice morphism, with $\cq$ finite and $\us{\psi(\cq)} = \cp$.
%    The map ${\floor -}_\psi : \cp \to \cq$ given by ${\floor p}_\psi \coloneqq \vee \{q \in \cq : \psi(q) \leq p \}$
%    is well-defined, a poset morphism, and a retraction for $\psi$, in the sense that ${\floor -}_\psi \circ \psi = \id_\cq : \cq \to \cq$.
%\end{lem}
\begin{proof}
    The map is well-defined since the set
    $\{q \in \cq : \psi(q) \leq p \}$ is finite and non-empty for every $p \in \cp$, since $\cq$ is finite, and for every $p \in \cp$ there exists $q \in \cq$ with $\psi(q) \leq p$.
    The map is order-preserving since $p \leq p' \in \cp$ implies
    $\{q : \psi(q) \leq p\} \subseteq \{q' : \psi(q') \leq p'\}$.

    We now show that ${\floor -}_\psi$ is the right adjoint of $\psi$, that is, that for every $q \in \cq$ and $p \in \cp$ we have $\psi(q) \leq p$ if and only if $q \leq \floor{p}_\psi$.
    If $\psi(q) \leq p$, then $q \in \{q' \in \cq : \psi(q') \leq p\}$, so $q \leq  \floor{p}_\psi$, by definition.
    If $q \leq \floor{p}_\psi$, then
    \[
        \psi(q) \leq \psi\left(\floor{p}_\psi\right)
                 = \psi\left(
                    \vee 
                    \left\{q' \in \cq : \psi(q') \leq p\right\}
                 \right)
                 =
                    \vee 
                    \left\{\psi(q') : \psi(q') \leq p\right\}
                \leq p,
    \]
    where we used the monotonicity of $\psi$, and the fact that it preserves non-empty finite joins.

    We conclude by proving that $\floor{\psi(q)}_\psi = q$ for every $q \in \cq$.
    Note that $q \in \{q' : \psi(q') \leq \psi(q)\}$, so it is sufficient to prove that every other $q' \in \cq$ such that $\psi(q') \leq \psi(q)$ satisfies $q' \leq q$, which follows from the fact that $\psi$ is full.
\end{proof}

As is standard, the adjunction $\psi \dashv {\floor -}_\psi$ allows us to give an explicit characterization of the left adjoint of the restriction functor $\res_\psi$ (this is used in, e.g., \cite[Sec.~4.3]{boo}, \cite[Thm.~2.7]{botnan-oppermann-oudot-scoccola}, or \cite[Prop.~7.12]{BBH2023}).

\begin{lem}
\label{lemma: section retraction between posets}
    Let $\psi : \cq \to \cp$ be a full, upper semilattice morphism, with $\cq$ finite and $\us{\psi(\cq)} = \cp$.
    The left adjoint $\Lan_\psi : \repq \to \repp$ of the restriction functor $\res_\psi :\repp \to \repq$ exists and is given by precomposition with ${\floor -}_\psi : \cp \to \cq$.
    The unit of the adjunction is induced by the equality ${\floor -}_\psi \circ \psi = \id_\cq : \cq \to \cq$.
    In particular $\Lan_\psi$ is exact and fully faithful.
\end{lem}
\begin{proof}
    The fact that a section-retraction pair of poset morphisms induces an adjunction between the corresponding restriction functors is a general, standard fact,
    so the first statement follows from \cref{lemma:retraction-for-f}.
    The left adjoint $\Lan_\psi$ is fully faithful since the adjunction unit is an isomorphism, and it is exact since it is given by precomposition.

    We conclude by showing that $\Lan_\psi$ takes values in finitely presented representations.
    Since $\Lan_\psi$ is exact, it is enough to observe that it maps the indecomposable projective $\K_{\us{q}} \in \rep \cq$ to the indecomposable projective $\K_{\us{\psi(q)}} \in \rep \cp$.
\end{proof}

We now prove that $\Lan_\psi$ preserves spread representations.
For an illustration, see \cref{figure:example-aligned-grid-inclusion-and-floor}.

\begin{lem}
    \label{lemma:stretching-of-spread-is-spread}
    Let $\psi : \cq \to \cp$ be a full, upper semilattice morphism, with $\cq$ finite, and $\us{\psi(\cq)} = \cp$.
    For all $A, B \subseteq \cq$, such that for every $b \in B$ there exists $a \in A$ with $a \leq b$, we have
    \[
        \Lan_\psi\left(\,\K_{\usminus{A}{B}}\,\right) \; \cong \; \K_{\usminus{\psi(A)}{\psi(B)}}. 
    \]
    In particular, the functor $\Lan_\psi : \repq \to \repp$ maps spread-decomposable representations to spread-decomposable representations.
\end{lem}
\begin{proof}
    The second statement follows from the first one and \cref{lemma:description-finitely-presented-spreads}.
    For the second statement, note that
    we have an exact sequence $\K_\us{B} \to \K_\us{A} \to \K_\usminus{A}{B} \to 0$.
    Since $\Lan_\psi$ is a left adjoint, it preserves cokernels, so it is sufficient to prove the statement for upset representations,
    that is, we need to prove that 
    $\Lan_\psi \K_{\us{S}} \cong \K_{\us{\psi(S)}}\in \rep(\cp)$, for every $S \subseteq \cq$.
    Since $\Lan_\psi$ is given by precomposition by $\floor{-}_\psi$ (\cref{lemma: section retraction between posets}),
    this is equivalent to proving that
    $\K_{\us{S}} \circ \floor{-}_\psi \cong \K_{\us{\psi(S)}}\in \rep(\cp)$,
    which in turn is equivalent to showing that, for every $p \in \cp$ 
    we have $p \in \us{\psi(S)}$ if and only if $\floor{p}_\psi \in \us{S}$.
    This last fact follows directly from the adjunction $\psi \dashv \floor{-}_\psi$
    (\cref{lemma: section retraction between posets}).
\end{proof}

%\justin{$f_!$ is only fully faithful when $f$ is fully faithful. I don't think join-preserving implies this. }

%\begin{defn}
%    \label{definition:stretching-functor}
%    Let $f : \cq \to \cp$ be an injective upper semilattice morphism with $\cq$ finite.
%    We refer to $f_! : \rep \cq \to \rep \cp$ of \cref{lemma: section retraction between posets} is the \emph{stretching functor} along $f$.
%\end{defn}

%The following remark is used implicitly in the results in the rest of this section.\justin{With my changes, we are very explicit. I included this remark for use in the last section.}
%
%\begin{rem}
%    Let $\cp$ and $\cq$ be bounded upper semilattices, with $\cq$ finite, and let $f : \cq \to \cp$ be a least element- and join-preserving, injective poset morphism.
%    In this case we have $\us{f(\cq)} = \cp$ so that, in particular, the retraction $\floor{-}_f$ is defined on the whole of $\cp$.
%\end{rem}

%In this subsection, we prove a key property of the stretching functor, namely, that it preserves spread-resolutions (\cref{prop: fus spread exact}).

%Our next goal is to prove that the left adjoint $\psi^!$ to the stretching functor $\psi^*$ maps spread-decomposable representations to spread-decomposable representations.
%We need the following two standard lemmas.

\subsection{Preservation of spread-approximations}
We first show a version of \cref{theorem:main-extension-theorem} in the case of finite grid posets (\cref{prop: fus spread exact}).
We then use this to prove \cref{theorem:main-extension-theorem}, from which \cref{theorem:answers-1-2-lattice} follows.
We state some results in the generality of upper semilattices, and only specialize to grid posets when needed.

\begin{lem}
    \label{lem:contraction is left adjoint}
    Let $\psi : \cq \to \cp$ be a full, upper semilattice morphism, with $\cp$ and $\cq$ finite, and $\us{\psi(\cq)} = \cp$.
    The functor $\Lan_\psi : \repq \to \repp$ admits a left adjoint $\Lan_{\floor{-}_\psi} : \repp \to \repq$.
\end{lem}
\begin{proof}
    The functor $\Lan_\psi$ is given by precomposition with $\floor{-}_\psi : \cp \to \cq$, by \cref{lemma: section retraction between posets}, so it preserves limits and thus admits a left adjoint, which can be computed using the colimit formula for left Kan extensions \cite[Ch.~6,~Cor.~2.6]{riehl}:
    \[
        \left(\Lan_{\floor{-}_\psi} M\right)(q)
        \;=\;
        \colim_{\substack{
                p \in \cp \\
                {\floor p} \leq q}}\,
            M(p),
    \]
    which is a finite-dimensional vector space, since $\cp$ is finite.
    %\justin{previously}\[
    %    \left(\Lan_{\floor{-}_\psi} M\right)(q)
    %    \;=\;
    %    \colim_{\substack{
    %            p \in \cp \\
    %            {\floor p} = q}}\,
    %        M(p),
    %\]
\end{proof}

\begin{lem}\label{lemma:floor-and-join}
    Let $\iota : \cq \to \cp$ be an aligned grid inclusion, with $\cq$ finite and $\us{\iota (\cq)} = \cp$.
    The right adjoint ${\floor -}_\iota : \cp \to \cq$ is an upper semilattice morphism.
\end{lem}
\begin{proof}
    It is enough to show that, given $x,y \in \cp$, we have $\floor{x \vee y}_\iota = \floor{x}_\iota \vee \floor{y}_\iota$.
    The result clearly holds for $\cp$ and $\cq$ totally ordered sets, since in this case $x \vee y = \max(x,y)$ and $\floor{p}_\iota = \max\{q \in \cq : \iota(q) \leq p\}$.
    For the general case, note that, in the case of grid posets, the operation $\floor{-}$ works component-wise (meaning that, for every $p = (p_1, \dots, p_n) \in \cp$, we have $\floor{p}_\iota = \left( \floor{p_1}_{\iota_1}, \dots, \floor{p_n}_{\iota_n} \right)$, where $\iota = \iota_1 \times \cdots \times \iota_n$), and 
    that join also works component-wise, so the result follows.
\end{proof}

\begin{lem}
    \label{lemma:initial-functor-floor-antichain}
    Let $\iota : \cq \to \cp$ be an aligned grid inclusion,
    with $\cq$ finite and $\us{\iota(\cq)} = \cp$.
    Let $A \subseteq \cp$, and let $M : \cq \to \kVec$.
    Precomposition with the poset morphism $\us{A} \to \us{{\floor A}_\iota}$
    given by restricting $\floor{-}_\iota$ to $\us{A}$ induces a natural isomorphism
    \[
        \lim_{\us{\floor{A}}}\, M
        \; \cong \;
        \lim_{\us{A}}\, \Lan_\iota M.
    \]
\end{lem}
\begin{proof}
    The functor $\Lan_\iota$ is given by precomposition with $\floor{-}_\iota$ (\cref{lemma: section retraction between posets}), so
    by a standard result in category theory (see the dual of \cite[Sec.~IX.3,~Thm.~1]{maclane}, or \cite[Lem.~3.8]{boo} for the case of poset morphisms), it is sufficient to show that $\floor{-}_{\iota}:\us{A} \to \us{{\floor A}_\iota}$ is an initial functor.
    That is, for every $q \in \us{{\floor A}_\iota}$, we need to show that the preimage
    \[
        B_q= \left\{ p \in \us{A} : \floor{p}_\iota \leq q\right\} \subseteq {\us{\iota \cq}}
    \]
    of the downset of $q$ is nonempty and connected.
    We start by proving that $B_q$ is non-empty.
    Since $q \in \us{{\floor A}_\iota }$, there exists $a \in A$ such that $\floor{a}_\iota \leq q$.
    Since $a \in \us{A}$, we have that $a \in B_q$.

    For connectivity, is enough to prove that $B_q$ is closed under joins.
    Let $x,y \in B_q$.
    On one hand, we have $y \vee x \in \us{A}$ since $\us{A}$ is an upset.
    On the other hand, we have $\floor{y \vee x}_\iota = \floor{x}_\iota \vee \floor{y}_\iota$, by \cref{lemma:floor-and-join}, and
    $\floor{x}_\iota \vee \floor{y}_\iota \leq q$ since 
    $\floor{x}_\iota \leq q$ and $\floor{y}_\iota \leq q$.
    Thus, $y \vee x \in B_q$, as required.
    %\justin{the following argument is false!} For connectivity, let $x,y \in B_q$.
    %Then $y \geq \psi({\floor y}_\psi) \leq \psi(q) \geq \psi({\floor x}_\psi) \leq x $ is a zigzag path from $y$ to $x$, as required.
\end{proof}

\begin{lem}\label{lemma:hom-upset-is-limit}
    Let $\cp$ be a poset and let $M : \cp \to \kvec$.
    There is an isomorphism of vector spaces $\lim M \cong \Hom_\cp(\K_{\cp}, M)$ which is natural in $M$. 
\end{lem}
\begin{proof}
    Recall that $\K_\cp : \cp \to \kvec$ is the constant functor at $\K$. 
    From the universal property of $\lim M$ \cite[Ch.~3,~Def.~1.5]{riehl},
    we get an isomorphism $\Hom_\cp(\K_\cp, M) \cong \Hom_\K(\K, \lim M)$, which is natural in $M$.
    The result then follows by combining this with the natural isomorphism
    $\Hom_\K(\K, \lim M) \cong \lim M$ given by considering the image of $1 \in \K$.
\end{proof}

The next result gives an explicit description of the action of the contraction functor on spread representations, extending known results such as \cite[Lem.~4.16]{boo}.
See \cref{figure:example-aligned-grid-inclusion-and-floor} for an example illustrating the result.

\begin{lem}\label{prop: fls respect spread}
    %Let $\psi : \cq \to \cp$ be a full, upper semilattice morphism, with $\cp$ and $\cq$ finite, and $\us{\psi(\cq)} = \cp$.
    Let $\iota : \cq \to \cp$ be an aligned grid inclusion,
    with $\cp$ and $\cq$ finite, and $\us{\iota(\cq)} = \cp$.
    For all $A, B \subseteq \cp$, such that for every $b \in B$ there exists $a \in A$ with $a \leq b$, we have
    \[
        \Lan_{\floor{-}_\iota}\left(\,\K_{\usminus{A}{B}}\,\right) \; \cong \; \K_{\usminus{\floor{A}_\iota}{\floor{B}_\iota}}. 
    \]
    In particular, the functor $\Lan_{\floor{-}_\iota} : \repp \to \repq$ maps spread-decomposable representations to spread-decomposable representations.
\end{lem}
    This proof is similar to that of \cref{lemma:stretching-of-spread-is-spread}, we do not know if there is a convenient common generalization.
\begin{proof}
    %The second statement follows from the first one, since every spread is of the form $\usminus{A}{B}$
    %(\cref{lemma:description-finitely-presented-spreads}).
    The second statement follows from the first one and \cref{lemma:description-finitely-presented-spreads}.
    For the second statement, note that
    we have an exact sequence $\K_\us{B} \to \K_\us{A} \to \K_\usminus{A}{B} \to 0$.
    Since $\Lan_{\floor{-}_\iota}$ is a left adjoint, it preserves cokernels, so it is sufficient to prove the statement for upset representations.
    To this effect, we use Yoneda's lemma, so that we need to prove that there exists a natural isomorphism
    \[
        \Hom_\cq\left(\K_\us{\floor{C}}, - \right)\; \cong\; \Hom_\cq\left(\Lan_{\floor{-}_\iota} \, \K_\us{C}, -\right).
    \]
    We use the following composite of natural isomorphisms
    \begin{align*}
        \Hom_\cq\left(\K_\us{\floor{C}_\iota}, M \right) 
            \;&\cong\;
            \lim_{\us{\floor{C}_\iota}} M
            \;\cong\;
            \lim_{\us{C}}\, \Lan_\iota M\\
            \;&\cong\;
            \Hom_\cp\left(\K_\us{C}, \Lan_\iota M\right)
            \;\cong\;
            \Hom_\cq\left(\Lan_{\floor{-}_\iota}\, \K_\us{C}, M\right),
    \end{align*}
    where, for the first and third isomorphisms, we use \cref{lemma:hom-upset-is-limit},
    for the second isomorphism, we use \cref{lemma:initial-functor-floor-antichain},
    and for the last isomorphism we use the adjunction $\Lan_{\floor{-}_\iota} \dashv \Lan_\psi$ of \cref{lem:contraction is left adjoint}.
\end{proof}

%\begin{figure}
%    \includegraphics[width=.8\linewidth]{figures/nonexample_floor_preserve_usminus.eps}
%    \caption{Above we see a spread $\usminus{a}{b}$ in the $3\times 4$ grid $\cp$. The crosses are the images of two different origin aligned grid inclusions $\iota$ and $\iota'$ respectively.
%    \emph{Left.} The image of $\iota$ contains the points $b$ and $q$, but does not contain $a$. Observe that $\floor{q}_{\iota} \in \usminus{\floor{a}_{\iota}}{\floor{b}_{\iota}}$. However, $q \not\in \usminus{a}{b}$, so $\floor{ q}_{\iota} \not\in \floor{\usminus{a}{b}}_{\iota}$.
%    %
%    \emph{Right.} The image of $\iota'$ contains the points $a$ and $q$, but does not contain $b$ or $p$. Observe that $\floor{p}_{\iota'} = \floor{b}_{\iota'}  $ is contained in $ \floor{\usminus{a}{b}}_{\iota'}$. However,  $\floor{p}_{\iota'} \in \us{\floor{b}_{\iota'} }$.
%    \luis{example of spread, of grid poset, aligned grid inclusion,\cref{lem: floor preserve usminus}}
%    }
%    \label{fig:nonexample-floor-preserve-usminus}
%\end{figure}

\begin{prop}\label{prop: fus spread exact}
    Let $\iota : \cq \to \cp$ be an aligned grid inclusion,
    with $\cp$ and $\cq$ finite, and $\us{\iota(\cq)} = \cp$.
    The functor $\Lan_\iota : \repq \to \rep{\cp}$ is exact, fully faithful, and maps spread-approximations to spread-approximations.
\end{prop}
\begin{proof}
    The functor $\Lan_\iota$ is exact and fully faithful by
    \cref{lemma: section retraction between posets}.
    Let $Y \to M$ be a spread-approximation, and let $X \in \rep{\cp}$ be spread-decomposable.
    Since $\Lan_\iota Y$ is spread-decomposable (\cref{lemma:stretching-of-spread-is-spread}) it is sufficient to show the following sequence is exact
    \[
        \Hom_{\cp}(X,\Lan_\iota Y) \to \Hom_{\cp}(X,\Lan_\iota M) \to 0\, ,
    \]
    which, by the adjunction $\Lan_{\floor{-}_\iota} \dashv \Lan_\iota$ is exact precisely when the following sequence is:
    \[
        \Hom_{\cq}(\Lan_{\floor{-}_\iota} X,Y) \to \Hom_{\cq}(\Lan_{\floor{-}_\iota} X,M) \to 0 \,.
    \]
    By \cref{prop: fls respect spread}, the representation $\Lan_{\floor{-}_\iota} X$ is spread-decomposable, so this last sequence is indeed exact, as required.
\end{proof}

%In this subsection we prove the following generalization of \cref{prop: fus spread exact} to the case where $\cp$ is not necessarily finite and the upset of the image of $\psi$ is not necessarily $\cp$.
%Its proof combines \cref{prop: fus spread exact} and \cite[Thm.~3.14]{aoki-escolar-tada}.

%\luis{must introduce finitely presented representations}
%\luis{must introduce finitely presented spreads and corresponding approximations}

\begin{defn}
    Let $\cp$ be a poset, and let $A \subseteq \cp$ be an upset, seen as a full subposet.
    The \emph{padding functor} $\pad_{A,\cp} : \rep A \to \rep \cp$ is defined on objects
    by mapping a representation $M : A \to \kvec$ to $\pad_{A,\cp} M : \cp \to \kvec$ with $(\pad_{A,\cp} M)(p) = M(p)$ if $p \in A$, and
    $(\pad_{A,\cp} M)(p) = 0$ if $p \in \cp \setminus A$.
    The structure morphisms of $\pad_{A,\cp} M$ and the action of $\pad_{A,\cp}$ on morphisms are the evident ones.
\end{defn}

The following is straightforward.
%, and essentially observed in \cite[Sec.~3.2]{aoki-escolar-tada} (see \cite[Lemma~3.13]{aoki-escolar-tada} and the discussion preceding it).

\begin{lem}
    \label{lemma:padding-functor-properties}
    Let $\cp$ be a poset, and let $A \subseteq \cp$ be an upset, seen as a full subposet.
    The left Kan extension $\rep A \to \rep \cp$ along the inclusion $A \subseteq \cp$ exists and is isomorphic to the padding functor $\pad_{A,\cp}$.
    This functor is exact and fully faithful,
    and preserves spread representations and finitely presented representations.
    \qed
\end{lem}

\begin{lem}
    \label{lemma:existence-stretching-infinite-case}
    Let $\iota : \cq \to \cp$ be an aligned grid inclusion, with $\cq$ finite.
    The functor $\Lan_\iota : \repq \to \rep{\cp}$ exists, and is given by 
    composing the functor $\Lan_\iota : \repq \to \rep \us{(\iota \cq)}$
    with the padding functor $\pad_{\us{(\iota \cq)}, \cp} : \rep \us{(\iota \cq)} \to \rep \cp$.
    In particular, the functor $\Lan_\iota$ is exact and fully faithful.
\end{lem}
\begin{proof}
    The functor $\Lan_\iota : \repq \to \rep \us{\iota (\cq)}$ exists by \cref{lemma: section retraction between posets}.
    Existence of $\Lan_\iota : \repq \to \rep{\cp}$ is then clear from the explicit description in the statement.
    Exactness and fully faithfulness follow from those of the padding functor and \cref{prop: fus spread exact}.
\end{proof}

Finally, we will need the following standard fact, which informally says that finite diagrams of finitely presented representations over an upper semilattice can be realized over a finite sub-upper semilattice.

\begin{lem}
    \label{lemma:restriction-to-finite-lattice}
    Let $\cp$ be an upper semilattice.
    Let $\mathcal{I}$ be a category with finitely many objects, and let $D : \mathcal{I} \to \rep \cp$ be a functor.
    There exists a finite subposet $\cp' \subseteq \cp$ such that $\cp'$ is a full upper semilattice, the inclusion $\psi : \cp' \to \cp$ is a full, upper semilattice morphism, and $D \cong \Lan_\psi \circ \res_\psi \circ D$.
    If $\cp$ is a grid poset, the subposet $\cp'$ can be taken to be a grid poset such that the inclusion $\cp' \to \cp$ is an aligned grid inclusion.
\end{lem}
\begin{proof}
    We describe the construction of $\cp'$; the fact that it satisfies the conditions in the statement is straightforward to verify.
    The values of $D$ on objects form a finite collection $\mathcal{M}$ of finitely presented representations of $\cp$.
    Let $G$ be the finite set of all distinct grades of generator and relations in the minimal presentations of the objects in $\mathcal{M}$.
    Then $\cp'$ is obtained by taking the closure of $G$ by joins.
    In the case where $\cp = \cp_1 \times \cdots \times \cp_n$ is a grid poset, the grid poset $\cp'$ can be taken to be $\pi_1(G) \times \cdots \times \pi_n(G) \subseteq \cp$.
\end{proof}

\mainextensiontheorem*
\begin{proof}
    Existence of the left adjoint $\Lan_\iota$, as well as the fact that it is exact and fully faithful, follows from
    \cref{lemma:existence-stretching-infinite-case}.

    Let $M \in \rep \cq$, and let $f : C \to M$ be a spread-approximation.
    We have that $\Lan_\iota M$ is spread-decomposable, by \cref{lemma:stretching-of-spread-is-spread}, and the fact that the padding functor preserve spread-decomposable representations.
    So it suffices to show that any morphism $g : \K_S \to \Lan_\iota M$ with $S \subseteq \cp$ a finitely presented spread factors through $f$ via a lift $\K_S \to \Lan_\iota C$.
    This lifting problem constitutes a finite diagram of finitely presented representations of $\cp$, so, by \cref{lemma:restriction-to-finite-lattice}, we may assume that $\cp$ is finite, and the only reason why we cannot directly apply \cref{prop: fus spread exact} is that $\us{\iota (\cq)}$ may be a proper subset of $\cp$.

    Factor $\iota : \cq \to \cp$ as $\cq \xhookrightarrow{\alpha} \us{\iota (\cq)} \xhookrightarrow{\beta} \cp$, and note that $\us{\iota (\cq)}$ is an upset in $\cp$ and an upper semilattice.
    Now, by \cref{prop: fus spread exact}, we have that $\Lan_\alpha f : \Lan_\alpha C \to \Lan_\alpha M$ is a spread-approximation.
    Note that the poset morphism $\beta$ is an inclusion of a convex, full subset,
    so, by \cref{lemma:padding-functor-properties}, 
    the left Kan extension $\Lan_\beta : \rep \us{\iota (\cq)} \to \rep \cp$ exists and is isomorphic to the padding functor $\rep \us{\iota (\cq)} \to \rep \cp$.
    Thus, by \cite[Thm~3.14]{aoki-escolar-tada}, the morphism $\Lan_\beta\left(\Lan_\alpha f\right) = \Lan_\iota f$ is also a spread-approximation.
    This implies the existence of the lift of $g$, as required.
%\luis{at this point, we could reprove the result from \cite{aoki-escolar-tada} we need, but let's consider this after the rest is done.
%In fact, if we want to be fully precise we need to explain more, since I think the functor used in \cite{aoki-escolar-tada} is the intermediate extension.}
\end{proof}

\finiteresolutionsinfinitecase*
\begin{proof}
    We start by proving that every finitely presented representation admits a finite spread-resolution.
    By \cref{lemma:restriction-to-finite-lattice,theorem:main-extension-theorem}, it is sufficient to prove that every representations of a finite poset admits a finite spread-resolution, which follows from \cite[Prop.~4.5]{AENY23}.

    We now prove that spread-projectives are precisely the spread-decomposable representations.
    Let $M$ be spread-projective.
    By the above, there exists a spread-approximation $C \to M$, which, by the fact that $M$ is projective, must admit a section $M \to C$.
    It follows that $M$ is a direct summand of $C$, and thus that it is spread-decomposable.
\end{proof}

\section{Bounding global dimension with radical approximations}
\label{section:bounding-relative-global-dimension}

In \cref{section:radical-approximations}, we recall the notion of radical approximation.
In \cref{section:radical-approximations-bound-global-dimension} we prove classical bounds on relative global dimension,
which we then use in \cref{section:main-abstract-bound} to bound from above the $\acx$-global dimension in the case where $\cx$ is obtained by applying a well-behaved family of functors to a class of modules $\cy$ over a different algebra.
These arguments use minimal radical approximations, and in order to use the results it is necessary to have an explicit description of minimal radical approximations, which we give in \cref{section:results-radical-approximations}.

\subsection{Radical approximations}
\label{section:radical-approximations}
Let $\cc$ be a linear Krull--Schmidt category.
A \textit{$\cc$-radical approximation} of $M \in \cc$ is a morphism $\rho: C \to M$ such that, for all $A\in \cc$, the image of $\Hom_\cc(A, \rho) : \Hom_\cc(A, C) \to \Hom_\cc(A,M)$ is $ \rad_\cc(A,M)$.
Equivalently, $\rho:C\to M$ is 
a $\cc$-radical approximation of $M$ if $\rho \in \rad_\cc(C,A)$ and 
for all $A\in \cc$ and all $g \in \rad_\cc(A,M)$, there exists a lift $\tilde g : A \to C$ making the following diagram commute:
\begin{equation*}
    \begin{tikzcd}
                         & A \arrow[d, "g"] \arrow[ld, "\tilde g"', dashed] \\
    C \arrow[r, "\rho"'] & M \,.
    \end{tikzcd}
\end{equation*}
A $\cc$-radical approximation $\rho : C \to M$ is \emph{minimal} if $\rho$ is right minimal.

\begin{rem}
The reason why we say ``$\cc$-radical approximation'' and not just ``radical approximation'' is that, in our applications, the category $\cc$ is a full subcategory of a larger category, as in \cref{section:relative-homological-algebra}.
\end{rem}

See \cref{figure: eg spread radapp} for an illustration of a spread-radical approximation.

\subsection{Classical bounds on global dimension}
\label{section:radical-approximations-bound-global-dimension}

We now prove a standard result about the relationship between the global dimension of a finite-dimensional algebra $\blam$ and the global dimension of $\blam$ relative to $\add{\cx}$.
We prove this for completeness, and also because these bounds are usually stated with the unnecessary assumption that $\proj\bgam\subseteq \add \cx$ (see \cref{remark:relative-precov-surjective}).
Before doing so, we prove a basic result relating radical approximations and projective presentations of relative simples,
which is a relative version of a well-known result in Auslander--Reiten theory \cite[Ch.~IV,~Lem.~6.4]{assem}.

\begin{lem}
    \label{lemma:relative-radical-approximation-presentation}
    In the projectivization setup (\cref{definition:projectivization-setup}), let $X \in \cx$.
    A morphism $\rho: C \to X$ is a (minimal) $\add\cx$-radical approximation if and only if
    the induced sequence $HC \to HX \to \coker {H\rho}$ is a (minimal) projective presentation of $S_X$ .
\end{lem}
\begin{proof}
    Assume that $\rho : C \to X$ is a (minimal) $\add\cx$-radical approximation.
    By the definition of radical approximation (\cref{section:radical-approximations}), the image of $H\rho =\Hom_\blam(G,\rho) : \Hom_\blam(G,C) \to \Hom_\blam(G,X)$ equals the radical $\rad_{\add\cx}(G,X)$, which in turn equals the radical $\rad_\blam(G,X)$, since $\add{\cx}$ is a full subcategory.
    So $\coker{H\rho} = HX / \rad_{\blam}(G,X) = S_X$.
    We have $\rad_{\blam}(G,X) = HX\, \circ\, \rad_\blam(G,G)$, and that $\rad_\blam(G,G)$ is the Jacobson radical of $\bgam$, by \cref{lem: cat rad properties}(2).
    Then, the morphism $H X \to \coker{H\rho} = S_X $ is a projective cover, by \cite[Ch.~I,~Thm.~4.4]{ARS}.
    This proves the forward implication.

    The converse is straightforward, using the fact that $H$ restricts to an equivalence between $\add{\cx}$ and $\proj{\bgam}$ (\cref{add G resolution thm}(3)).
\end{proof}

%\luis{adjust this paragraph}
%For the reasons described in \cref{remark:relative-precov-surjective}, \cref{lem:gldim-lambda-leq-gldim-gamma} is often stated with the extra assumption that $\proj\bgam\subseteq \add \cx$. 
%We get the following bounds on the relative global dimension.
%We need the following standard result, which says that the global dimension of a finite dimensional algebra is attained at a simple.
%\begin{lem}[{\cite[Ch.~I,~Prop.~5.1]{ARS}}]
%    \label{lemma:global-dimension-attained-simple}
%    If $\bgam$ is a finite dimensional algebra, then 
%    \[
%    \pushQED{\qed}
%    \gldim{}{\bgam} \;\;=\; \max_{\substack{S \in \fmod \bgam\\ \text{simple}}}\;  \pdim{\bgam}{S}\,. \qedhere
%    \popQED
%    \]
%\end{lem}

\begin{lem}
    \label{lem:bound-global-dimension-radical-approximations}
    In the projectivization setup (\cref{definition:projectivization-setup}), we have
    \[
        \gldim{}{\bgam} \; \leq \; \max_{X\in \cx}\; \pdim{\add\cx}{\ker \rho_X} + 2\,,
    \]
    where $\rho_X : C_X \to X$ is an $\acx$-radical approximation, for every $X \in \cx$.
\end{lem}
\begin{proof}
    By the standard fact that the global dimension of a finite-dimensional algebra is attained at a simple \cite[Ch.~I,~Prop.~5.1]{ARS}, we have
    \[
        \gldim{}{\bgam} = \max_{X\in \cx}\; \pdim{\bgam}{S_X} .
    \]
    Now, note that $H \rho_X$ is a minimal projective presentation of $S_X$, by \cref{lemma:relative-radical-approximation-presentation}.
    Since~$H$ is left exact, we have $\ker (H \rho_X) \cong H(\ker \rho_X)$.
    Let $Z_\bullet \to \ker \rho_X$ be a minimal $\acx$-resolution.
    By \cref{add G resolution thm}(3), the following is a minimal projective resolution of $S_X$:
    \begin{equation*}
        \begin{tikzcd}
        \cdots \arrow[r, "H z_1"] & H Z_0 \arrow[r, dashed] \arrow[rd, "H z_0"] & H \rho_X \arrow[r, "H\rho_X"] & H X \arrow[r] & S_X \\
                                  &                                             & H(\ker \rho_X) \arrow[u, hook]  &                          &                
        \end{tikzcd}
    \end{equation*}
    It follows that $\pdim{\bgam}{S_X} \leq 2+\pdim{\bgam}{\ker H \rho_X}$, so that
    \begin{align*}
        \gldim{}{\bgam}\; &=\; \max_{X\in \cx}\; \pdim{\bgam}{S_X} \\
                &\leq\; \max_{X\in \cx}\; \pdim{\bgam}{\ker H \rho_X} + 2 \\
                &=\; \max_{X\in \cx}\; \pdim{\add\cx}{\ker \rho_X} + 2\, ,
    \end{align*}
    where in the second equality we used \cref{add G resolution thm}(5), concluding the proof.
\end{proof}

The following bounds for relative global dimension are standard \cite[Lem.~2.1]{ehis04}.

\begin{prop}
    \label{proposition:standard-gldim-bounds}
    In the projectivization setup (\cref{definition:projectivization-setup}), we have:
    \[
          \gldim{\add \cx}{\fmod \blam}
          \;\leq\;
          \gldim{}{\bgam}
          \;\leq\;
          \gldim{\add \cx}{\fmod \blam}+ 2 \,.
    \]
    If, moreover, we have $\inj\blam \subseteq \add\cx$, then
    \[
        \gldim{}{\bgam}
        \;= \;
        \max_{X\in \cx}\; \pdim{\add\cx}{\ker \rho_X} + 2
        \;= \;
        \gldim{\add \cx}{\fmod \blam} + 2\,,
    \]
    where $\rho_X$ is a minimal $\acx$-approximation of $X$ for every $X \in \cx$.
\end{prop}
\begin{proof}
    \noindent\emph{First inequality.}
    This follows directly from \cref{add G resolution thm}(5).

    \medskip

    \noindent\emph{Second inequality.}
    For each $X \in \cx$, let $\rho_X : C_X \to X$ be a minimal $\cx$-radical approximation. 
    We can then compute as follows
    \begin{align*}
        \gldim{}{\bgam} &\leq \max_{X\in \cx}\; \pdim{\add\cx}{\ker \rho_X} + 2 \\
                &\leq \gldim{\add \cx}{\fmod \blam} + 2\,,
    \end{align*}
    Where in the first inequality we used \cref{lem:bound-global-dimension-radical-approximations}, and in the second we used the definition of global dimension.

    \medskip

    \noindent\emph{Equalities.}
    By the inequalities above, it is enough to show that 
    $\gldim{\add \cx}{\fmod \blam} + 2 \leq \gldim{}{\bgam}$.
    Let $M \in \fmod \blam$, and let $M \to I_0 \xrightarrow{c} I_1$ be a $\blam$-injective copresentation.
    The induced exact sequence $0 \to HM \to H I_0 \to H I_1 \to \coker H c \to 0$ implies that
    \[
        \pdim{\bgam}{HM} \leq \max \left\{
            \pdim{\bgam}{H I_0},
            \pdim{\bgam}{H I_1}-1,
            \pdim{\bgam}{H c}-2
              \right\},
    \]
    which simplifies to $\pdim{\bgam}{HM} \leq \max \left\{0, \pdim{\bgam}{H c}-2 \right\}$,
    since $H I_0$ and $H I_1$ are $\bgam$-projective, by
    \cref{add G resolution thm} and the assumption that $\inj{\blam} \subseteq \add{\cx}$.
    This implies that 
    $\pdim{\bgam}{HM} + 2 \leq \pdim{\bgam}{H c} \leq \gldim{}{\bgam}$, so that
    $\pdim{\add{\cx}}{M} + 2 \leq \gldim{}{\bgam}$, by \cref{add G resolution thm}(5).
\end{proof}

The following remark relates our constructions to those in \cite{asashiba24}.

\begin{rem}\label{rem:comparision-with-asashiba}
    In the projectivization setup (\cref{definition:projectivization-setup}), let $X \in \cx$.
    In the proof of \cref{proposition:standard-gldim-bounds}, we spliced a minimal $\cx$-radical approximation of $\rho_X : C_X \to X$ with a minimal $\add{X}$-resolution $Z_\bullet \to \ker \rho_X$ to obtain a minimal projective resolution $\cdots \to H Z_1 \to H Z_0 \to H C_X$ of the relative simple $S_X$.
    The corresponding sequence $\cdots \to Z_1 \to Z_0 \to C_X$ obtained using \cref{add G resolution thm}(3) is isomorphic to the Koszul resolution $\mathscr{K}_\bullet(X)$ of $X$ relative to $\add\cx$, whose dual is defined in \cite[Def.~3.2]{asashiba24}.
    This follows from \cite[Thm.~7]{asashiba24}, which says that $H \mathscr{K}_\bullet(X)$ is a minimal projective resolution of $S_X$.
    The interesting fact is that, a priori, $H \mathscr{K}_\bullet(X)$ is different from the projective resolution we use, and the isomorphism between them is a posteriori.
    This also means that the explicit description of radical approximations in \cref{prop: sum irr rad app} can be used to give an alternative construction the Koszul co-resolutions of \cite{asashiba24}.
\end{rem}

\subsection{Main technical bound on global dimension}
\label{section:main-abstract-bound}

The goal of this subsection is to prove the following result.

\begin{prop}\label{thm: big small radapp}
    Suppose that we are given the following:
    \begin{itemize}
    \item Two finite-dimensional algebras $\Delta$ and $\blam$. 
    \item A finite set $\cx$ of pairwise non-isomorphic indecomposable $\Lambda$-modules.
    
    \item A finite set $\cy$ of pairwise non-isomorphic indecomposable $\Delta$-modules. 
    \item A finite set $\Phi$ of additive functors $\fmod\Delta \to \fmod\blam$. 
    \end{itemize}
    Assume that the following are satisfied:
    %\luis{1 and 2 are 1 now}
    %\luis{3 is 2}
    %\luis{5 is 3}
    %\luis{4 is 4}
    \begin{enumerate}
        \item Each $\phi \in \Phi$ is exact and fully faithful.
        %\item Each $\phi \in \Phi$ is exact. 
        \item Each $\phi \in \Phi$ maps $\acy$-approximations to $\acx$-approximations.
        \item Each $\phi \in \Phi$ restricts to a functor $\add\cy \to \add\cx$.
        \item For every $X \in \cx$, there exists $\phi_X \in \Phi$ such that $\add{\phi_X \cy}$ contains $X$ and the domain of a minimal $\add{\cx}$-radical approximation of $X$.
    \end{enumerate}
    Then, we have
    \begin{align*}
          \gldim{\add \cx}{\fmod \blam} &\leq 
          \gldim{\add\cy}{\fmod \Delta} \, +2\, .
    \end{align*}
    If, moreover, we have $\inj\bgam \subseteq \add\cx$ and $\inj\Delta \subseteq \add\cy$, then
    \begin{align*}
          \gldim{\add \cx}{\fmod \blam} &= 
          \gldim{\add\cy}{\fmod \Delta} \, .
    \end{align*}
\end{prop}

We briefly describe the main ideas that go into the proof of \cref{thm: big small radapp}.
Assumptions (1), (3), and (4) ensure that every $\acx$-radical approximation can be constructed by applying some $\phi_X \in \Phi$ to some $\acy$-radical approximation, 
and this is the content of \cref{prop: domain rad app iso}.
Assumptions (1), (2), and (3) ensure that each $\phi \in \Phi$ maps $\acy$-resolutions to $\acx$-resolutions,
and this is the content of \cref{cor: ker rad app iso}.

%\begin{rem}
%    \luis{The following is already being said in the overview of approach section, \cref{section:overview-of-approach}}
%    Our main result \cref{theorem:stabilization} is proven by constructing a solution to the following homological algebra problem:
%    Given a finite set $\cx$ of pairwise non-isomorphic indecomposable $\blam$-modules. Find a finite set $\cy$ of indecomposable $\Delta$-modules and a set of functors $\Phi$, satisfying the hypothesis of \cref{thm: big small radapp}.
%    %In \cref{thm: bound spread gl dim}, $\blam$ is the incidence algebra of some finite grid $\cp$
%\end{rem}

%We need the following two properties satisfied under the hypotheses of \cref{thm: big small radapp}.

\begin{lem}\label{prop: domain rad app iso}
    Under the hypotheses of \cref{thm: big small radapp},
    %In the relative cocovering setup (\cref{def:small-big-setup}), let $X \in \cx$.
    for every $X \in \cx$,
    %, let $Y \in \cy$.
    there exists a functor $\phi_X \in \Phi$ and a minimal $\add{\cy}$-radical approximation $r_X : D \to Y$
    such that $\phi_X Y \cong X$ and $\phi_X r_X : \phi_X D \to \phi_X Y \cong X$ is a minimal $\add{\cx}$-radical approximation.
    %There exists $\phi \in \Phi$ and a minimal $\add{\cy}$-radical approximation $\rho_Y : C_Y \to Y$ such that $\phi \rho_Y : \phi C_Y \to \phi Y$ is a minimal $\add{\cx}$-radical approximation.
\end{lem}
\begin{proof}
    By assumption (4) of \cref{thm: big small radapp}, there exists a minimal $\add{\cx}$-radical approximation $\rho_X : C_X \to X$ and a functor $\phi_X\in \Phi$ such that $\add{\phi_X \cy}$ contains $X$ and $C_X$.
    %\justin{I changed this to accomodate the assumptions of \cite[Ch.~I,~Prop.~2.1]{ARS}}
    %Let $\rho_X : C_X \to X$ be an $\add{\cx}$-radical approximation.
    By assumptions (1) and (3), each $\phi_X \in \Phi$ restricts to an equivalence of categories $\acy \to \add{\phi_X \cy}$, so there exists $\phi_X \in \Phi$, and $r_X : D \to Y$ in $\add \cy$ such that $\phi_X r_X \cong \rho_X$.
    It suffices to show that $r_X$ is a minimal $\acy$-radical approximation.

    Let $\rho_Y : C_Y \to Y$ be a minimal $\acy$-radical approximation.
    Note that $\rho_Y$ is a right minimal morphism, by definition, and that $r_X$ is also a right minimal morphism, by \cref{lem: fully faithful preserve minimality}.
    So, in order to prove that $r_X$ and $\rho_Y$ are isomorphic, it is enough to prove that $r_X$ factors through $\rho_Y$, and that $\rho_Y$ factors through $r_X$, by \cite[Ch.~I,~Prop.~2.1]{ARS}.
    %To show that $r_X$ factors through $\rho_Y$, we use the fact that $\rho_Y$ is a radical approximation, and to show that $\rho_Y$ factors through $r_X$, we use the fact that $\phi_X r_X$ is a radical approximation, so that $\phi_X \rho_Y$ factors through $\phi_X r_X$, which implies that $\rho_Y$ factors through $r_X$, since $\phi_X$ is fully faithful.
    
    For the first factorization, \cref{lem: fully faithful preserve minimality} ensures that $r_X \in \rad_{\add\cy}(D,Y)$. Since $\rho_Y$ is a radical approximation, we have $r_X$ factors through $\rho_Y$. 
    Similarily, \cref{lem: fully faithful preserve minimality} ensures that $ \phi_X \rho_Y ~\in \rad_\blam( \phi_X C_Y, \phi_X Y) $. By assumption (3), $\phi_X \rho_Y$ is contained in $\acx$, so $\phi_X \rho_Y$ factors through $\phi_X r_X$, because $\rho_X = \phi_X r_X $. Thus, $\rho_Y$ factors through $r_X$, since $\phi_X$ is fully faithful.
\end{proof}

\begin{lem}\label{cor: ker rad app iso}
    Under the hypotheses of \cref{thm: big small radapp}, each $\phi \in \Phi$ maps (minimal) $\acy$-resolutions to (minimal) $\acx$-resolutions. In particular, for all $M \in \fmod\Delta$ we have $\pdim{\acy}{M} = $ $\pdim{\acx}{\phi M}$.
\end{lem}
\begin{proof}
   Suppose that $R_\bullet \to M$ is a minimal $\acy$-resolution in $\fmod\Delta$.
   By assumption (1), each $\phi \in \Phi$ preserves kenerls.
   By assumption (2), $\phi$ preserves $\acx$-approximations of kernels, so $\phi R_\bullet$ is an $\acx$-resolution of $\phi M$.
   Moreover, the resolution $\phi R_\bullet$ is minimal, by \cref{lem: fully faithful preserve minimality}, so the result follows.
\end{proof}

\begin{proof}[Proof of \cref{thm: big small radapp}]
    Fix, for each $Y \in \cy$, a minimal $\acy$-radical approximation $\sigma_Y : C_Y \to Y$, and
    for each $X \in \cx$, a minimal $\acx$-radical approximation $\rho_X : C_X \to X$.
    
    By \cref{prop: domain rad app iso}, for each $X \in \cx$, there exists $\phi_X \in \Phi$ and a minimal $\acy$-radical approximation $r_X$ such that $\phi_X r_X \cong \rho_X$,
    and by assumption (1) of \cref{thm: big small radapp}, we have $\ker \phi_X r_X \cong \phi_X \ker r_X$.
    Using this and \cref{cor: ker rad app iso}, we get
    \begin{align*}
        \pdim{\add\cx}{\ker \rho_X} &= \pdim{\add\cx}{\ker \phi_X r_X}\\
                                    &= \pdim{\add\cx}{\phi_X \ker r_X}\nonumber\\
                                &= \pdim{\add\cy}{\ker r_X},
    \end{align*}
    
    Let $\bgam \coloneqq \End_\blam\left(\bigoplus_{X \in \cx} X\right)$, so that we are in the projectivization setup (\cref{definition:projectivization-setup}).
    To prove the first claim in \cref{thm: big small radapp}, we bound as follows:
    \begin{align*}
            \gldim{\acx}{\fmod \blam}&\leq \gldim{}{\bgam}
            \\&\leq \max_{X\in \cx}\; \pdim{\add\cx}{\ker \rho_X} + 2
            \\&= \max_{X\in \cx}\; \pdim{\add\cy}{\ker r_X} + 2
            \\&= \max_{Y\in \cy}\; \pdim{\add\cy}{\ker \sigma_Y} + 2
            \\&\leq \gldim{\acy}{\fmod \Delta} + 2
            \,.
    \end{align*}
    The first line in this computation follows
    from \cref{proposition:standard-gldim-bounds}. The second line follows
    from \cref{lem:bound-global-dimension-radical-approximations}. The third line follows from %\cref{eq:pdim-ler-big-small}. 
    the argument in the second paragraph.
    The fifth line is immediate from the definition of relative global dimension (\cref{section:relative-homological-algebra}). 
    For the fourth line, let $Y\in \cy$. Then $X'=\phi Y \in \cx$ has a minimal $\acx$-radical approximation $\rho_{X'}= \phi_{X'} r_{X'} $. 
    Since $r_{X'}$ is a minimal $\acy$-radical approximation of $Y$, we have $\sigma_Y \cong r_{X'}$. It follows that the fourth line is less than or equal to the third. The reverse inequality is trivial.

    To prove the second claim, assume that $\inj\bgam \subseteq \add\cx$ and $\inj\Delta \subseteq \add\cy$.
    In this case, the first inequality in the chain above is an equality,
    and $\gldim{}{\bgam}$ also equals $\gldim{\acx}{\fmod \blam} + 2$, thanks to \cref{proposition:standard-gldim-bounds}.
    So it is enough to prove that the last inequality in the chain above is also an equality, which also follows from \cref{proposition:standard-gldim-bounds}.
\end{proof}

\subsection{Explicit minimal radical approximations}
\label{section:results-radical-approximations}

We give a remark about nomenclature.

\begin{rem}\label{rem:comparison-with-CM}
    Let $\ca$ be an additive subcategory of $\fmod \blam$ for some finite-dimensional algebra~$\blam$.
    In Definition~2 of \cite{chalom2006irreducible}, on which we rely, an $\ca$-radical approximation is referred to as a \textit{$\ca$-right almost split morphism}. 
    The reason for this is that the $\fmod \blam$-radical approximations of an indecomposable $C \in \fmod \blam$ are exactly the right almost split morphisms ending at $C$ in the usual sense described in \cite[Ch. IV,~ Def.~1.1]{assem}.
    It is a classical result in Auslader--Reiten theory \cite[Ch.~VII,~Lem.~1.2]{ARS}, that irreducible morphisms determine the almost split morphisms. Taking $\cx$ to be the entire module category in \cref{prop: sum irr rad app} reproduces this classical fact.
\end{rem}

We now give an explicit construction of relative radical approximations.
%We now give an explicit construction of relative radical approximations, which
%is a relative version of a well-known result in Auslader--Reiten theory \cite[Ch.~VII,~Lem.~1.2]{ARS}.

%\justin{\cref{lem: cat rad properties}}
%\luis{what does the comment right before this one mean?}

\begin{prop}
    \label{prop: sum irr rad app}
    In the projectivization setup (\cref{definition:projectivization-setup}), let $X \in \cx$. 
    \begin{itemize}
        \item For each $Y\in \cx$, choose a set $\{f_{Y,\ell} \}_{\ell \in L_Y}$ of $\acx$-irreducible morphisms, such that the set of equivalence classes $\left\{\overline{f}_{Y,\ell}\right\}_{\ell \in L_Y}$ is a basis for ${\rad_{\add\cx}(Y,X)}/{\rad_{\add\cx}^2(Y,X)}$ as a module over the division ring ${\Hom_{\blam}(Y,Y)}/{\rad_\blam(Y,Y)}$. 
        \item Define
            $C_\cx(X)\coloneqq \bigoplus_{Y\in \cx} \bigoplus_{\ell \in L_Y} Y $.
        \item Define
            $\rho_\cx(X) \coloneqq \sum_{Y \in \cx} \sum_{\ell \in L_Y} f_{Y,\ell} : C_\cx(X) \to X$.
            %$\rho_Y \coloneqq [\cdots\,f_{Y,\ell} \,\cdots] : \bigoplus_\ell Y \to X$, and define
            %$\radapp_\cx X : C_\cx(X) \to X$ to be the morphism $[\cdots\,\rho_Y \,\cdots] : C_\cx(X) \to X$.
    \end{itemize}
    The morphism $\rho_\cx(X) : C_\cx(X) \to X$ is a minimal $\add\cx$-radical approximation of $X\in \cx$.
\end{prop}
\begin{proof}
    By \cref{cor:approx exist krull}, any module $X\in \cx$ admits a minimal $\add\cx$-radical approximation.
    The statement then follows from \cite[Ch.~2,~Thm.~1]{chalom2006irreducible}.
\end{proof}

\begin{comment}
\begin{defn}
    \label{definition:map-of-spreads}
    Let $S,T \subseteq \cp$ be spreads of a poset $\cp$.
    We say that $T$ \emph{maps nicely to} $S$ if $S\cap T$ is a downset of $T$ and an upset of $S$.
    In that case, we define a morphism $\onemorphism_{T,S} : \K_T \to \K_S$ that is the identity over $S \cap T$ and zero elsewhere.
\end{defn}

It is straightforward to see that the morphism $\onemorphism_{T,S}$ of \cref{definition:map-of-spreads} is well-defined.
\end{comment}

Combining \cref{prop: sum irr rad app} with the description of the spread-irreducible morphisms in \cite{BBH2023} allows us to describe the domain of the minimal spread-radical approximation of spread representations.

\begin{prop}\label{prop: spread radical approx domain}
    Let $\cp$ be a finite poset, let $S \subseteq \cp$ be a spread, and let
    $\rho : C \to \K_S$ be a minimal spread-radical approximation.
    The representation $C$ is isomorphic to the direct sum of spread representations $\K_T$ where $T$ ranges over the following set of spreads:
    \begin{align*}
        \big\{\, V \, & :\, V = S \sqcup (\dsminus{x}{S}) \text{ for some } x \in \cover S \,\big\} \\
        \sqcup\,
        \big\{\, W \, & :\, S = W \sqcup (\usminus{s}{W}) \text{ for some } s\in \cocover W \,\big\}.
    \end{align*}
\end{prop}
\begin{proof}
    By \cite[Prop.~6.8]{BBH2023}, a morphism $f:\K_T \to \K_S$ is spread-irreducible if and only if one of the following is true (see \cref{figure: eg spread radapp} for an illustration):
    \begin{enumerate}
        \item The morphism $f$ is surjective, and there exists $x\in\cover{S}$ such that $T = S \sqcup (\dsminus{x}{S})$. In this case, the spread $T$ belongs to the first set in the union.
        
        \item The morphism $f$ is injective and there exists $s\in \minimal S$ such that $S = T \sqcup (\usminus{s}{T})$ and $T$ is a connected component of $S\setminus s$.
        By \cref{lem: cocover vs connected}, this occurs if and only if $s \in \cocover T$ and $T\in \spread \cp$. 
        In this case, the spread $T$ belongs to the second set in the union.
    \end{enumerate}
    By \cref{prop: sum irr rad app},
    the multiplicity of $\K_T$ as a direct summand of $C$ is the dimension of
    %the indecomposable summands of $C$ are in one to one correspondence with a basis of 
    $\rad_{\spread\cp}(\K_T, \K_S) / \rad^2_{\spread\cp}(\K_T,\K_S)$
    over $\hom_\cp(\K_S,\K_S)/\rad_\cp(\K_S,\K_S)$.
    %It is thus enough to show that this basis has cardinality at most one.
    %Note that \cite[Prop.~5.10]{BBH2022} gives a basis a for $\Hom_\cp(\K_U,\K_V)$ whenever $U, V \in \spread\cp$. 
    If $f\in \Hom_\cp(\K_T,\K_S)$ is spread-irreducible, then \cite[Prop.~5.10]{BBH2022} implies that $\Hom_\cp(\K_T,\K_S) \cong \K$.
    This means that $\rad(\K_T, \K_S) / \rad^2(\K_T,\K_S)$ has a basis of size one when $\Hom_\cp(\K_T,\K_S)$ contains a spread-irreducible, and of size zero otherwise, as required.
    %, and the indecomposable summands are pairwise nonisomorphic by \cref{prop: sum irr rad app}.
\end{proof}

\begin{figure}
    \includegraphics[width=\textwidth]{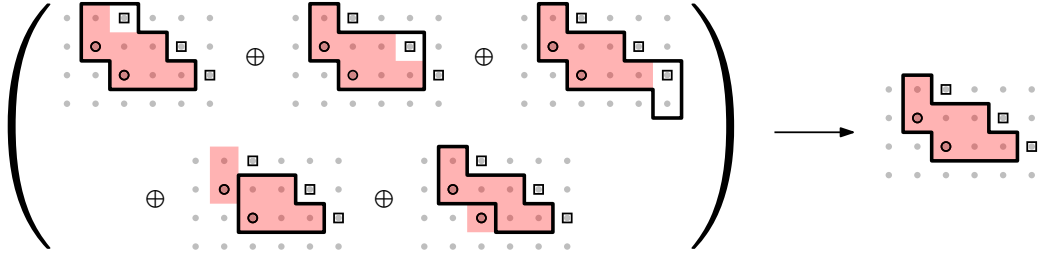}
    \caption{%
    A spread-radical approximation over a grid poset $[6] \times [4]$.
    The five spreads on the left constitute the summands in the domain of the radical approximation.
    The summands in the top row correspond to surjective spread-irreducibe morphisms, and each one corresponds to a cover of the spread.
    The summands in the bottom row correspond to injective spread-irreducible morphisms,
    and each corresponds to a minimum of the spread.
    The outline of the spread, as well as its minimal elements and covers are reproduced in the domain for comparison.}
    \label{figure: eg spread radapp}
\end{figure}

%\begin{eg}\todo{combine this with \cref{eg: proj spread irr}}\todo{move \cref{eg: proj spread irr} here}
%     Let $S$ be a spread in a finite poset $\cp$.
%     
%     $T^w$ does not need to exists, for example take $S=0$ in the 2x2 grid. There are no injections $\K_T \into \K_S$ with $T\neq S$ a spread.
%     This isn't an issue since the right hand summand of $\radapp_\cp S$ can be 0
%     
%     $T^w$ also does not need to be unique. for example consider $S=\cp$, where $\cp$ is the poset $x \leftarrow w \to y $. Then $\cp = x \cup (\us w \setminus x) = y \cup (\us w \setminus y) $. In this example $C_\cp S = \K_x \oplus \K_y$.
%\end{eg}

We conclude this section by giving two explicit computations of spread-radical approximations, for intuition.
One is in \cref{figure: eg spread radapp}, and the other is the next example.

\begin{eg}
    Let $\cp$ be the poset with underlying set $\{a,b,c\}$, and with order relation generated by $b \leq a$ and $b \leq c$.
    Consider the following four representations of $\cp$
    \[
        \K \xleftarrow{\id_\K} \K \xrightarrow{\id_\K} \K\,,
        \;\;\;
        \K \leftarrow 0 \rightarrow \K\,,
        \;\;\;
        \K \leftarrow 0 \rightarrow 0\,,
        \;\;\;
        0 \leftarrow 0 \rightarrow \K\,,
    \]
    which are 
    isomorphic to $\K_{\us{b}}$,
    $\rad \K_{\us{b}}$,
    $\K_{\us{a}}$,
    and
    $\K_{\us{c}}$, respectively.
    Note that $\rad \K_{\us{b}} \cong \K_{\us{b} \setminus b} \cong \K_{\us{a}} \oplus \K_{\us{b}}$.

    Since $\K_\us{b}$ is projective, every surjection $\K_T \to \K_\us{b}$ admits a section. 
    Therefore, there are no spread-irreducible surjections ending at $\K_\us{b}$, which is consistent with the fact that the spread $\us{b}$ has no covers.
    This accounts for case (1) in
    \cite[Prop.~6.8]{BBH2023} and in the proof of \cref{prop: spread radical approx domain}.

    The spread $\us{b}$ has unique minimum $b$. The spread-irreducible injections ending at $\K_{\us{b}}$
    correspond to the connected components of $\us{b} \setminus b$, and are given by the
    canonical inclusions $\K_{\us{a}} \to \K_{\us{b}}$ and $\K_{\us{c}} \to \K_{\us{b}}$.
    This accounts for case~(2) in
    \cite[Prop.~6.8]{BBH2023} and in the proof of \cref{prop: spread radical approx domain}.
\end{eg}
%\justin{haven't defined succesor}

\section{Stabilization of the spread-global dimension}
\label{section:stabilization-dimension}

\subsection{Spreads in the image of the left Kan extension}

We start by giving sufficient conditions for a spread to be in the image of the left Kan extension along an upper semilattice morphism.
Before doing this, we need a lemma.

\begin{lem}
    \label{lem: floor preserve usminus}
    Let $\psi : \cq \to \cp$ be a full, upper semilattice morphism, with $\cp$ and $\cq$ finite, and $\us{\psi(\cq)} = \cp$.
    If $A, B \subseteq \psi(\cq)$,
    then $\us{\floor{A}_\psi} \setminus \us{\floor{B}_\psi} = \floor{ \us{A} \setminus \us{B} }_\psi$.
\end{lem}
\begin{proof}
    For convenience, we denote $\floor{-}_\psi$ by $\floor{-}$.
    We first prove that $\us{\floor{A}} \setminus \us{\floor{B}} \subseteq \floor{ \us{A} \setminus \us{B} }$.
    Let $\floor p \in \floor{\usminus{A}{B}}$.
    Since $p \in \us A$, there exists $a \in A$ with $\floor a \leq \floor p$. So $\floor p \in \us{\floor A}$. 
    Suppose for the purpose of contradiction that $\floor b \leq \floor p$ for some $b\in B$. Since $b \in \psi \cq$ we have $b = \psi \floor b \leq \psi\floor p \leq p$, which contradicts the assumption that $p \not\in \us B$.  

    We now prove the other inclusion.
    Let $q \in \usminus{\floor A}{\floor B}$.
    Then there exists $a\in A \subseteq \psi \cq$ with $a = \psi\floor a \leq \psi q$. Suppose for the purpose of contradiction that $b \leq \psi q$. Then $\floor b \leq \floor{\psi q} = q$ which is a contradiction. Therefore, $\psi q \in \usminus{A}{B}$ and $q =\floor{\psi q} \in \floor{\usminus{A}{B}}$.
\end{proof}

%In the next result, the left Kan extension $\Lan_\psi$ is that guaranteed to exist by \cref{theorem:main-extension-theorem}.
%Moreover, note that, in the setup of the result, the morphism $\floor{-}_\psi$ is defined on $\us{\psi(Q)} \subseteq \cp$, and thus $\floor{S}_\psi$ is well-defined.

\begin{prop}\label{propopsition: every spread hit by floor}
    Let $\psi : \cq \to \cp$ be a full, upper semilattice morphism, with $\cp$ and $\cq$ finite, and $\us{\psi(\cq)} = \cp$.
    %Let $\psi : \cq \to \cp$ be an upper semilattice morphism, with $\cq$ and $\cp$ finite,
    Let $S \subseteq \cp$ be a spread.
    If the image of $\psi$ contains 
    $\min S$ and $\cover S$, then
    \[
    \K_S \cong
        \Lan_\psi \left(\K_{\floor{S}_\psi}\right)\; \in\; \rep\cp \,,
    \]
    so in particular $\K_S$ is in the essential image of $\Lan_\psi : \repq \to \repp$.
\end{prop}
\begin{proof}
    %Since $S \subseteq \us{\psi(Q)}$, we may assume that $\cp = \us{\psi(Q)}$.
    We compute as follows
    \begin{alignat*}{3}
            \K_S
        &=
            \K_{(\us{\min S})
            \setminus
            (\us{\cover S})}
        && \text{\footnotesize (\cref{lemma:description-finitely-presented-spreads})}
        \\
        &\cong
            \K_{\us{\psi\left(\floor{\min S}_\psi\right)}
            \setminus
            \us{\psi\left(\floor{\cover S}_\psi\right)}}
        && \text{\footnotesize (since $\min S \subseteq \psi(\cq)$ and $\cover S \subseteq \psi(\cq)$)}
        \\
        &\cong
            \Lan_\psi\left(
                \K_{\us{\floor{\min S}_\psi}
                \setminus
                \us{\floor{\cover S}_\psi}}\right)
        && \text{\footnotesize (\cref{lemma:stretching-of-spread-is-spread})}
        \\
        &=
            \Lan_\psi\left(
                \K_{\floor{\us{\min S}
                            \setminus
                           \us{\cover S}}_\psi}\right)
        && \text{\footnotesize (\cref{lem: floor preserve usminus})}
        \\
        &=
            \Lan_\psi\left(
                \K_{\floor{S}_\psi}
            \right)
        && \text{\footnotesize (\cref{lemma:description-finitely-presented-spreads})}
    \end{alignat*}
    \par \vspace{-1.5\baselineskip}
    \qedhere
    \smallskip
    %the result follows from the explicit description of the action of $\Lan_\psi$ on spread representations of \cref{lemma:stretching-of-spread-is-spread}.
    %\luis{Here there's a short argument missing}
\end{proof}

The following notations make the next results easier to state and prove.

\begin{notation}
    \label{notaion:plus-minus}
    We equip any finite grid poset $\cg$ with addition and subtraction operations defined as follows.
    If $\cg = \ct_1 \times \cdots\times \ct_n$ is a finite grid poset,
    we identify each $\ct_i$ with $[m_1]$ for some $m_i \in \N$,
    and define
\begin{align*}
    \pi_i(x+ y) = \min \{ m_i-1,~ \pi_i(x) + \pi_i(y)\}
    \;\;\;\;
    \text{and}
    \;\;\;\;
    \pi_i(x- y) = \max \{ 0,~ \pi_i(x) - \pi_i(y)\}\,,
\end{align*}
for all $x,y \in \cg$, where $\pi_i$ is the $i$th projection (\cref{section:posets-of-interest}).
We also denote by $e_1,\dots , e_n$ the standard basis vectors in $\Z^{n+1}$,
so that, if $B \subseteq \cg$ and $0 \leq i \leq n$, we let 
\begin{align*}
    B + e_i \coloneqq \left\{b+e_i : b \in B \right\} \subseteq \cg
    \;\;\;\;
    \text{and}
    \;\;\;\;
    B - e_i \coloneqq \left\{b-e_i : b \in B \right\} \subseteq \cg\,.
\end{align*}
When we consider a product $\ct \times \cg$ between a totally ordered set $\ct$ and a grid poset $\cg$, which is again a grid poset, we let $\pi_0 : \ct \times \cg \to \ct$ represents the projection onto the factor~$\ct$ (as in \cref{section:posets-of-interest}).
\end{notation}

The next result gives an explicit subgrid satisfying the assumptions of 
\cref{propopsition: every spread hit by floor}.

\begin{lem}
    \label{prop the plus 1 grid}
    Let $\ct$ be a finite total order, let $\cg$ be a finite grid poset, and let $S \subseteq \ct \times \cg$ be a spread.
    If $x$ is a cover of $S$, then $\pi_0(x) \in \pi_0 (\min S \cup (\max S +e_0))$. 
    In particular, the aligned subgrid
    $\pi_0 (\min S \cup (\max S +e_0)) \times \cg$ contains $\min S$ and $\cover S$.
\end{lem}
%The finite grid $\cp$ is a join semi-lattice with join (least upper bound) given by \[ \pi_i\left(\bigvee A \right)~=~ \max \{\pi_i(a) ~|~ a\in A \}~.\]
\begin{proof}%[Proof of \cref{prop the plus 1 grid}]
    It is clear that the second claim follows from the first.
    Let $A=\min S$, $B= \max S$, and $\cq = \pi_0(A\cup (B+e_0) )\times \cg$.
    %Clearly $A=\min S \subseteq\cq$, so the second claim follows from the first. 
    %Reordering the $\ct_i$, it is sufficient to show $\pi_0(x) \in \pi_0(A \cup (B+e_0))$. 
    There are three cases for a cover $x=(x_0,\dots, x_n)$ of $S$
    (for an illustration of the first and second case, see the points $x$ and $y$ in \cref{figure: spread lemmas}, respectively).
    
    \medskip \noindent \emph{Case 1.}
    Suppose $\pi_0(x) \neq 0$ and $x- e_0 \not\in \ds S$. By the definition of cover, there exists $a=(a_0,\dots,a_n) \in A$ with $a \leq x$.
    If $a \leq x - e_0 \leq x$ then $x - e_0 \in S \subseteq \ds S$, which is a contradiction. 
    Therefore $a \not\leq x - e_0$. 
    Now $ a \leq x$, so $a_0 \leq x_0$ but $a_0 \not\leq x_0 -1$. 
    Then $ x_0 -1 < a_0 \leq x_0$ and $a_0 = x_0 = \pi_0(x)$. Since $\pi_0(x) \in \pi_0( A) $, we have $x\in \cq$.

    \medskip \noindent \emph{Case 2.}
    Suppose $\pi_0(x) \neq 0$ and $x- e_0 \in \ds S =\ds B$, meaning there exists $b\in B$ with
    \begin{equation*}
        x- e_0 = (x_0 -1, x_1,\dots, x_n) \leq (b_0,b_1,\dots, b_n) = b\,.
    \end{equation*}
    Thus $x_0 -1 \leq b_0$. 
    If $x_0 \leq b_0$, then $x \leq b$, which contradicts the fact that $x \not\in \ds S$. 
    Therefore $x_0 -1 \leq b_0 < x_0$ so $b_0 +1 = x_0 = \pi_0(x)$. Since $\pi_0(x) \in \pi_0( B + e_0) $, we have $x\in \cq$.

    \medskip \noindent \emph{Case 3.}
    Suppose $\pi_0(x) =0$. Since $x$ is a cover, there exists an $a\in A$ with $a\leq x$. Then $\pi_0(a) \leq \pi_0(x) =0$ so $\pi_0(x) =\pi_0(a) =0$. Thus $\pi_0(x) \in \pi_0(A)$, and $x \in \cq$.
\end{proof}

The next example shows that, in nontrivial cases, the subgrid constructed in \cref{prop the plus 1 grid} can be the smallest possible that allows one to apply \cref{propopsition: every spread hit by floor}.

\begin{eg}\label{eg: plus 1 grid}
    Consider the spread $S =\usminus{\iota A}{\iota B}$ in the finite grid poset $\cp$ of \cref{figure:example-aligned-grid-inclusion-and-floor}.
    By \cref{lemma:description-finitely-presented-spreads}, we have $\min S = \iota A$ and $\cover S = \iota B$, so $ \usminus{\floor{\iota A}_\iota}{\floor{\iota B}_\iota} = \floor{\usminus{\iota A}{\iota B}}_\iota $, by \cref{lem: floor preserve usminus}.
    The image $\iota \cq = \pi_0(\min S \cup (\max S + e_0))$ has the form described in \cref{prop the plus 1 grid}, so $\K_S \cong \Lan_\iota\left( \K_{\floor{S}_\iota}\right)$, by \cref{propopsition: every spread hit by floor}.
    In this case, the sets $\minimal S$ and $\cover S$ are distinct maximal antichains (in the sense of \cref{lemma: antichain bound}). 
    The subgrid $\iota \cq \subseteq \cp$ is the smallest origin 
    aligned subgrid satisfying the assumptions of \cref{propopsition: every spread hit by floor}.
    %Intuitively, \cref{propopsition: every spread hit by floor} ensures $S$ has the same "shape" as $\floor S$ with .
\end{eg}

\begin{figure}
    \includegraphics[width=.6\linewidth]{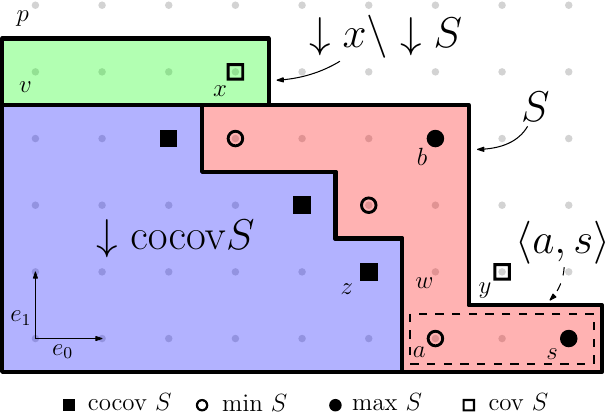}
    \caption{
    The spread $S$ in $\cp =[9]\times [7]$ serves as an example throughout \cref{section:stabilization-dimension,section:technical-results-spreads}.
    We also consider the spreads $V \coloneqq S \sqcup (\dsminus{x}{S})$ and $W \coloneqq S \setminus \iprod{a}{s}$, corresponding to cases (1) and (2) of \cref{prop: spread radical approx domain}, respectively.}
    \label{figure: spread lemmas}
\end{figure}

We conclude by constructing an explicit subgrid that allows us to apply \cref{propopsition: every spread hit by floor} to hit all spreads in a fixed spread-radical approximation (see the proof of \cref{theorem:stabilization}).

\begin{prop}\label{prop rad approx hit by floor}
    Let $\ct$ be a finite total order, let $\cg$ be a finite grid poset,
    and let $S = \iprod{A}{B} \subseteq \ct \times \cg$ be a spread.
    The following subgrid
    contains all the minimal elements and the covers of $S$,
    as well as the minimal elements and the covers of all spreads $T$ such that $\K_T$ is a direct summand of the domain of a minimal spread-radical approximation of $\K_S$:
    \[
        \pi_0\big( 0 \cup A\cup (A+ e_0) \cup (B+ e_0) \cup (B+2  e_0) \big) \times \cg \; \subseteq \; \ct \times \cg\,.
    \]
\end{prop}

\begin{proof}
    Let $\cq$ be the subgrid from the statement. \cref{prop the plus 1 grid} shows that $\cq$ contains the minimal elements and covers of $S$. 
    \cref{prop: spread radical approx domain} describes the summands $\K_T$ of the radical approximation. In particular, the second claim reduces to verifying that $\cq$ contains the following:
    \begin{enumerate}
        \item The minimal elements and the covers of $\iprod{A}{B} \sqcup (\dsminus{x}{B})$, for every $x \in \cover{S}$.
    \item The minimal elements and the covers of every spread $W \subseteq \cp$,
    for which there exists $a \in A$ with $\iprod{A}{B} = W \sqcup (\usminus{a}{W})$ and with $a \in \cocover{W}$.
    \end{enumerate}

    We first check (1).
    The maximal elements of $V= \iprod{A}{B} \sqcup (\dsminus{x}{B})$ are contained in $x \cup B$
    (e.g., consider the maxima of $V = S\sqcup(\usminus{x}{S})$ in \cref{figure: spread lemmas}).
    By \cref{prop the plus 1 grid}, the covers of $V$ are contained in the poset
    $\cq_V =\pi_0 \big( A \cup (B \cup x) +e_0 \big) \times \cg$,
    and $\pi_0(x) \in \pi_0( A \cup (B+ e_0))$.
    Thus,
    \[ \cover{V}\subseteq \cq_V \subseteq \pi_0\big( A \cup (B+ e_0) \cup (A +e_0)\cup (B+ 2e_0) \big) \times \cg = \cq\,.\]
    (e.g., in \cref{figure: spread lemmas}, the covers of $V $ are $p$, $y$, and $x+e_0$, and they satisfy $\pi_0(p)=0$, $\pi_0(y)\in \pi_0(B+e_0)$, and $\pi_0(x+e_0) \in \pi_0(A+e_0)$).

    Next we show $\min V \subseteq \cq$.
    If $V$ is a downset, then $\min V = 0$ is contained in $\cq$.
    Suppose now that $V$ is not a downset, so that, equivalently, we have $\cocover V\neq \emptyset$, by the definition of cocover.
    Then, the minimum $0 \in \ct \times \cg$ belongs to $\ds{\cocover{V}}$, and thus
    $\minimal V \subseteq \cover{(\ds{\cocover V})}$, by \cref{lem: min spread implies cover ds}.
    It is therefore sufficient to show that $\cq$ contains the covers of the spread $Z:= \ds{\cocover V}$.
    (e.g., in \cref{figure: spread lemmas}, we have $\min V = v\sqcup \min S$, and
    the minimum $v$ is a cover of $\ds{\cocover V}$).
    Now, \cref{prop the plus 1 grid} implies that
    \[
        \pi_0(\minimal V) \subseteq \pi_0(\cover{Z}) \subseteq \pi_0( \min{Z} \cup (\max{Z}+e_0))\,.\tag{$\ast$}
    \]
    By \cref{lemma:description-finitely-presented-spreads}, the set $Z$ is the downset of an antichain. Therefore, $\min Z =0$ and by \cref{lemma: cocover of spread minus cover} $\max Z={\cocover V}={\cocover S}$.

    %By \cref{lemma: cocover of spread minus cover}, we have $\ds{\cocover T} = \ds{\cocover S} $, which is a spread with maxima $\cocover S$ and unique minimum $0$.
    The addition and subtraction operations of \cref{notaion:plus-minus} are duals of each other (in the sense of \cref{remark:poset-duality}).
    %Furthermore, there is a dual version of \cref{prop the plus 1 grid} for cocovers. 
    The dual of \cref{prop the plus 1 grid} shows that
    $\pi_0(\max Z) =\pi_0(\cocover S) \subseteq \pi_0( (A -e_0) \cup B)$.
    Combining this with ($\ast$), we have
    \[
        \pi_0(\minimal V) \subseteq \pi_0( \min{Z} \cup (\max{Z}+e_0)) \subseteq \pi_0( 0\cup A \cup (B+e_0))\, .
    \]
    %(e.g., consider the minimum $v\in \min V$ in \cref{figure: spread lemmas}, which is a minimal element of $V$ such that $\pi_0(v)=0$).
    It follows that $ \min V \subseteq \pi_0( 0\cup A \cup (B+e_0))\times \cg \subseteq\cq$.
    
    \smallskip

    We now check (2).
    The dual of \cref{lemma: cocover of spread minus cover} implies that $\cover{S} = \cover{W}$, and we have already shown that $\cover{S}\subseteq \cq$, by \cref{prop the plus 1 grid}. 
    (e.g., if $W:= S\setminus \iprod{a}{s}$ in \cref{figure: spread lemmas}, then $\cover W = \{x,y\} = \cover S$). 

    It remains to show $\min W \subseteq \cq$.
    By \cref{lem: cocover vs connected}, we have that $W$ is a connected component of $\iprod{A}{B}\setminus a$. 
    Therefore, if $w \in \min{W}$, then either $w\in A$ or $w=a+e_i$ for some $0\leq i \leq n$ (e.g., in \cref{figure: spread lemmas}, we have $w=a+e_1 \in \min W$). 
    We thus get that $\pi_0(\min{W})\times \cg \subseteq \pi_0\big( A\cup (A+ e_0) \big) \times \cg$, so $\cq$ contains the minimal elements of $W$. 
\end{proof}

\subsection{Proof of \cref{theorem:stabilization} and consequences}

Before we give the proof of \cref{theorem:stabilization},
we require a standard result.

\begin{lem}\label{lemma: antichain bound}
    If $\cq$ is a finite poset, and $\ct$ is a finite total order,
    then the largest antichain in $\ct \times \cq$ has cardinality at most $|\cq|$.
\end{lem}
\begin{proof}
    Consider the projections $\pi_\ct : \ct \times \cq \to \ct$ and $\pi_\cq : \ct \times \cq \to \cq$.
    If $A \subseteq \ct \times \cq$ and the restriction $\pi_\cq|_A : A \to \cq$ is not injective, then there exist $a \neq b \in A$ with $\pi_\cq(a) = \pi_\cq(b)$ and $\pi_\ct(a) \leq \pi_\ct(b)$ (since $\ct$ is a total order), and thus $A$ contains two distinct comparable elements.
    This implies that, if $A$ is an antichain, then 
    $\pi_\cq|_A$ is injective, and we have $|A| \leq |\cq|$.
\end{proof}

\stabilization*
\begin{proof}
    Let us start by proving that the second statement implies the first.
    Any finite poset~$\cq$ admits a full embedding into a finite grid poset $\cg$ (indeed, into $\cg = \{0 < 1\}^\cq = [2]^{\cq}$, by mapping $q \in \cq$ to $\{x \in \cq : x \leq q\}$).
    This, in turn, induces a full embedding of $\ct \times \cq$ into $\ct \times \cg$, so that
    \[
        \gldim{\spread}{\ct \times \cq} \leq
        \gldim{\spread}{\ct \times \cg} \leq n_\cg = \gldim{\spread}{[k] \times \cq},
    \]
    where the first inequality is due to the monotonicity of the spread-global dimension \cite[Thm.~1.2]{aoki-escolar-tada},
    the second inequality is by definition of $n_\cg$, 
    and the equality is by hypothesis.
    The global dimension $\gldim{\spread}{[k] \times \cq}$ is finite \cite[Prop.~4.5]{AENY23} and independent of $\ct$, so $n_\cq$ is finite.

    \medskip

    To prove the second statement, it is enough to prove that
    \[
        \gldim{\spread}{[\ell] \times \cg} \leq \gldim{\spread}{[k] \times \cg}
    \]
    for every $\ell \geq 1 \in \N$, where $k = 1 + 4 \cdot |\cg|$.
    If $\ell \leq k$, this follows from the monotonicity of the spread-global dimension \cite[Thm.~1.2]{aoki-escolar-tada}, so let us assume that $\ell > k$.

    We use \cref{thm: big small radapp} with $\Delta = \K ([k] \times \cg)$ and $\blam = \K([\ell] \times \cg)$ poset algebras so that $\fmod \Delta \simeq \rep ([k] \times \cg)$ and
    $\fmod \blam \simeq \rep ([\ell] \times \cg)$ (\cref{remark:poset-representations-as-modules}),
    and with $\cy$ being of the spread representations of $[k] \times \cg$ and $\cx$ being the spread representations of $[\ell] \times \cg$.
    If the hypothesis of \cref{thm: big small radapp} hold, then we have $\gldim{\spread}{[\ell] \times \cg} = \gldim{\spread}{[k] \times \cg}$, since in this case, both $\add \cx$ and $\add \cy$ contain the corresponding injective objects, by \cref{lemma:proj-is-spread}.

    It is thus enough to construct a finite set of additive functors
    $\rep([k] \times \cg) \to \rep([\ell] \times \cg)$
    satisfying the following conditions:
    \begin{enumerate}
        \item All of the functors are exact and fully faithful.
        \item All of the functors map spread-approximations to spread-approximations.
        \item All of the functors map spread-decomposable representations to spread-decomposable representations.
        \item For every spread $S\subseteq [\ell] \times \cg$ there exists one of the functors whose essential image contains $\K_S$ and the domain of a minimal spread-radical approximation of $\K_S$.
    \end{enumerate}
    As set of functors, we choose
    \[
        \left\{\,\Lan_\iota :  \text{$\iota$ is an origin aligned grid inclusion $[k] \times \cg \hookrightarrow [\ell] \times \cg$}\,\right\},
    \]
    where the left Kan extensions are guaranteed to exist by \cref{theorem:main-extension-theorem} (or \cref{prop: fus spread exact}), using the fact that any aligned grid inclusion is an upper semilattice morphism (\cref{lemma:aligned-grid-inclusion-semilattice-morphism}).
    This set of functors satisfies conditions (1) and (2), by \cref{theorem:main-extension-theorem}, as well as condition~(3), by \cref{lemma:stretching-of-spread-is-spread}.

    To conclude, we must show that this set of functors also satisfies condition (4).
    That is, we must show that, for every spread $S \subseteq [k] \times \cg$, there exists an aligned grid inclusion $\iota : [\ell] \times \cg \to [k] \times \cg$ such that the essential image of $\Lan_\iota$ contains $\K_S$ and the domain of a minimal spread-radical approximation of $\K_S$.
    Let $S = \iprod{A}{B}$ with $A$ and $B$ antichains,
    define $E = 0 \cup A\cup (A+ e_0) \cup (B+ e_0) \cup (B+2  e_0)$
    and consider the subgrid
    \[
        \pi_0(E) \times \cg \subseteq [k] \times \cg.
    \]
    By \cref{prop rad approx hit by floor,propopsition: every spread hit by floor},
    it is enough to construct an aligned grid inclusion
    $\iota : [k] \times \cg \to [\ell] \times \cg$ such that $\pi_0(E) \times \cg  \subseteq \iota([\ell] \times \cg)$.
    In order to construct this,
    it is enough to produce a morphism $[k] \to [\ell]$ with image containing 
    $\pi_0( E)$.
    This morphism exists as long as 
    $|E| \leq k$, so let us check this.

    Since $A$ and $B$ are antichains, we have
    $|A| \leq |\cg|$ and $|B| \leq |\cg|$, by \cref{lemma: antichain bound}.
    We can then bound $|E| = |0 \cup A\cup (A+ e_0) \cup (B+ e_0) \cup (B+2  e_0)| \leq 1 + 2 |A| + 2 |B| \leq 1+ 4 \cdot |\cg| = k$,
    as required.
\end{proof}

\begin{proof}[Proof of \cref{corollary:conjecture-1}]
    Let $g_2(k) \coloneqq \gldim{\spread}{ [k] \times [2] }$.
    By \cref{theorem:stabilization}, we have $g_2(k) \leq g_2(1 + 4 \cdot 2) = g_2(9)$, for every $k \geq 1 \in \N$.
    So, by the monotonicity of the spread-global dimension \cite[Thm.~1.2]{aoki-escolar-tada}, we get that $g_2$ is constant for $k \geq 9$.
    The result then follows from the fact that $g_2(4) = \dots = g_2(9) = 2$,
    by \cite[Ex.~4.10]{AENY23}.
\end{proof}

\begin{proof}[Proof of \cref{corollary:conjecture-2}]
    By \cref{theorem:stabilization}, we have $g_m(k) \leq g_m(1 + 4 m)$ for every $k \geq 1 \in \N$.
    The result then follows from the monotonicity of the spread-global dimension \cite[Thm.~1.2]{aoki-escolar-tada}.
\end{proof}
    
\begin{proof}[Proof of \cref{corollary:R-times-lattice}]
    If $M \in \rep(\ct \times \cg)$, then, by \cref{lemma:restriction-to-finite-lattice},
    there exists a finite total order $\ct' \subseteq \ct$ and a representation $M' \in \rep(\ct' \times \ct)$, such that 
    $M \cong \Lan_\iota M'$, where $\iota : \ct' \times \cg \to \ct \times \cg$ is the corresponding aligned grid inclusion. 
    Take a minimal spread-resolution of $M'$, which has length at most 
    $n_\cq < \infty$, by \cref{theorem:stabilization}.
    By \cref{theorem:main-extension-theorem},
    applying $\Lan_\iota$ to this spread-resolution we obtain a spread-resolution of $M$ of length at most $n_\cq$, as required.
\end{proof}

%%%%%%%%%%%%%%%%%
\appendix

\section{Technical results about spreads}
\label{section:technical-results-spreads}

\begin{defn}
    An upset $U \subseteq \cp$ of a poset $\cp$ is \emph{finitely generated}
    if there exists a finite set $A$ with $\us{A} = U$.
\end{defn}

\begin{lem}
    \label{lemma:finitely-generated-upset}
    If $\cp$ be a poset and $U \subseteq \cp$ is a finitely generated upset, then $\min U$ is finite and $U = \us{\min U}$.
\end{lem}
\begin{proof}
    Suppose $\us{A} = U$ with $A$ finite. 
    If $u \in \min U$ then there exists $a\in A \subseteq U$ with $a\leq u$, and thus $a=u$, by minimality.
    This implies that $\min U \subseteq A$, so $\min U$ is finite. 
    Since $\us{\min U} \subseteq U$, to conclude it is enough to show that $U \subseteq \us{\min U}$.
    Given $u \in U$, there exists $a \in A$ with $a \leq u$.
    Since $A$ is finite, we may choose $a$ to be a minimal element of $A$ with this property.
    If $a$ is a minimal element of $U$, then $a \in \min U$, and we are done, so let us show that $a$ must be a minimal element of $U$.
    If this were not the case, then there would exist $U \ni u' < a$, but then $u' \notin \us{A}$, by the minimality of $a$, which is a contradiction.
\end{proof}

\begin{lem}
    \label{lemma:finitely-presented-implies-finitely-generated-upsets}
    If $\cp$ is a poset and $S \subseteq \cp$ a finitely presented convex set, then $\us{S}$ and $\us{S} \setminus S$ are finitely generated upsets.
\end{lem}
\begin{proof}
    Since $\K_S \in \rep \cp$ is finitely presented, there are functions
    $\beta_0 : \cp \to \N$ and $\beta_1 : \cp \to \N$ of finite support, and a minimal projective presentation
    \[
        \bigoplus_{j \in \cp} \left(\K_{\us{j}}\right)^{\beta_1(j)}
        \to
        \bigoplus_{i \in \cp} \left(\K_{\us{j}}\right)^{\beta_0(i)}
        \to \K_S\,.
    \]
    The second morphism is surjective, so that $S \subseteq \us{\supp(\beta_0)}$.
    The second morphism is, moreover, a projective cover,
    so that $\supp(\beta_0) \subseteq S$.
    It follows that $\us{\supp(\beta_0)} = \us{S}$.
    The statement for $\us{S}$ then follows from \cref{lemma:finitely-generated-upset}
    by taking $A = \supp(\beta_0)$.

    Similarly, we must have $\us{S} \setminus S = \us{\supp(\beta_1)}$,
    so the statement for $\us{S} \setminus S$ also follows from \cref{lemma:finitely-generated-upset}
    by taking $A = \supp(\beta_1)$.
\end{proof}

\begin{proof}[Proof of \cref{lemma:description-finitely-presented-spreads}]
    Taking $U = \us{S}$ in \cref{lemma:finitely-generated-upset}, we get that $\min{\us{S}}$ (which also equals $\min{S}$) is finite, and that $\us{S} = \us{\min{(\us{S})}} = \us\min S$.
    Taking $U = \us{S} \setminus S$ in
    \cref{lemma:finitely-generated-upset}, we get that $\min(\us{S} \setminus S) = \cover(S)$ is finite, and that $\us{S} \setminus S = \us{\min(\us{S} \setminus S)} = \us{\cover{S}}$.
\end{proof}

\begin{proof}[Proof of \cref{cor: cover iff min upminus}]
    Since $S$ is convex, we have $S = \ds S \cap \us S$, and
\begin{align*}
    \cp& = S \sqcup (\us S \setminus S) \sqcup (\ds S \setminus S) \sqcup \big(\cp \setminus(\ds S \cup \us S )\big)\,.
\end{align*}
Note that $\cp \setminus(\ds S \cup \us S ) = \incomp S$, by definition.
We get $\us S \setminus S = \us{\cover S}$ by using 
\cref{lemma:finitely-generated-upset,lemma:finitely-presented-implies-finitely-generated-upsets}.
Dually, we also get $\ds S \setminus S = \ds{\cocover S}$, as required.
%\luis{Finish this proof \cref{lemma:description-finitely-presented-spreads}.
%}
%\luis{What was here:}
%    Since $S$ is convex, it is a downset in $\us S$.
%    Thus, $\us S \setminus S$ is an upset. If $\cp$ is finite, then $\minimal(\us S \setminus S)= \cover{S}$ by \cref{lemma:description-finitely-presented-spreads}.
%    Thus, we have $\us S = S \sqcup \us{\cover S} $, and, by duality, also $\ds S = S \sqcup \ds{\cocover S}$.
%    To conclude, we substitue into
%    the equality above.
\end{proof}

In the next results, we denote a spread $W$ and $V$ (rather than $S$) in order to match the notation in the proof of \cref{prop rad approx hit by floor}, where these results are used.
The notation $V$ corresponds to case (1) of \cref{prop: spread radical approx domain}, while the notation $W$ corresponds to case (2).

For an illustration of the next result, let $S$ and $T = S \sqcup (\dsminus{x}{S})$ be as in \cref{figure: spread lemmas}.

\begin{lem}\label{lemma: cocover of spread minus cover}
    Let $\cp$ be a finite poset, and let $V \subseteq \cp$ be a spread.
    If there exists $x \in \cover S$ such that $V = S \sqcup (\dsminus{x}{S})$,
    then $\cocover V = \cocover S$
\end{lem}
\begin{proof}
    By the dual of \cref{lemma:description-finitely-presented-spreads}, it is sufficient to show that $\ds V \setminus V = \ds S \setminus S$. 
    Since the downset operator commutes with unions, we have $\ds V = \ds S \sqcup \ds(\dsminus{x}{S}) = \ds S \cup \ds x$, so that
    \begin{align*}
        \ds V \setminus V & = \big(\ds S \cup \ds x \big) \setminus \big( S\sqcup (\dsminus{x}{S}) \big)\,.
    \end{align*}
    Now, a straightforward set theoretic argument shows that $\ds V \setminus V = \ds S \setminus S$, as required.
\end{proof}

\begin{eg}
    The spread $V = S\sqcup (\usminus{x}{S})$ in \cref{figure: spread lemmas} gives examples of \cref{lemma: cocover of spread minus cover,lem: min spread implies cover ds}.
    The cocovers of $S$ (also in \cref{figure: spread lemmas}) are also cocovers of $V$, as proven in \cref{lemma: cocover of spread minus cover}. 
    Furthermore, we have that $\ds{\cocover V} = \ds{\cocover S}$ contains $0=\min \cp$.
    The covers of $\ds{\cocover S} =\ds{\cocover V} $ are $v\sqcup \min S $, where $v \in \incomp S$ is the unique minimum of ${\dsminus{x}{S}}$.
\end{eg}

\begin{lem}
    \label{lem: min spread implies cover ds}
    Let $\cp$ be a finite poset and let $V \subseteq \cp$ be convex.
    If $\min\cp \subseteq \ds{\cocover V} $, then 
    $\minimal V \subseteq \cover{(\ds{\cocover{V}})}$.
\end{lem}
\begin{proof}
Let $Z \coloneqq \ds{\cocover{V}}$, so that
$\cover{Z} = \minimal (\us Z \setminus Z)$, by definition of cover. 
Since $\min\cp \subseteq Z$, we have $\cp= \us{\min\cp} \subseteq \us{Z}$, so $\us Z \setminus Z = \cp \setminus Z= \cp\setminus \ds{\cocover{V}} $.
To conclude, note that
\[
    \cover{(\ds{\cocover{V}})} = \min (\us Z\setminus Z) = \minimal (\cp\setminus \ds{\cocover{V}}) = \minimal(\incomp V) \sqcup \minimal V\,,
\]
where in the last equality we used that $\cp\setminus \ds{\cocover{V}}= (\incomp V) \sqcup \us{V}$, by \cref{cor: cover iff min upminus}.
\end{proof}
\begin{comment}
\begin{proof}
Let $Y \coloneqq (\us{\ds{\cocover{V}}}) \setminus \ds{\cocover{V}} $, so that
$\cover{(\ds{\cocover V})} = \minimal (Y)$, by definition of cover. 
Since $\min\cp \subseteq \ds{\cocover V}$, we have $\cp= \us{\min\cp} \subseteq \us{\ds{\cocover V}}$, so $Y = \cp\setminus \ds{\cocover{V}} = \cp \setminus \left(\ds{V} \setminus V\right)$.
To conclude, note that
\[
    \cover{(\ds{\cocover{V}})} = \min Y = \minimal (\cp\setminus \ds{\cocover{V}}) = \minimal(\incomp V) \sqcup \minimal V\,,
\]
where in the last equality we used that $Y= (\incomp V) \sqcup \us{V}$, by \cref{cor: cover iff min upminus}.
\end{proof}
\end{comment}

\begin{lem}\label{lem: cocover vs connected}
    Let $\cp$ be a finite poset, let $W \subseteq \cp$ be a spread, and let $a \in \ds W \setminus W$.
    We have $a \in \cocover W$ if and only if $W$ is a connected component of $W \sqcup \left((\usminus{a}{W}) \setminus a\right)$.
\end{lem}
\begin{proof}
    \noindent ($\Rightarrow$)
    Since $W$ is a spread, it is connected by assumption, so it suffices to prove that there is no zigzag path between an element of $W$ and an element of $(\usminus{a}{W}) \setminus a$ entirely contained in $W \sqcup \left((\usminus{a}{W}) \setminus a\right)$.
    Towards a contradiction, assume that there is a zigzag path between an element $(\usminus{a}{W}) \setminus a$ and an element of $W$, so that there is a $s \in (\usminus{a}{W}) \setminus a$ and $w \in W$ that are comparable.
    If $w\leq s$, then $s \in \us W$, which contradicts the fact that $s \notin W$.
    If $s \leq w$, then,
    using that $s \in (\usminus{a}{W}) \setminus a$, we get that
    $a < s \leq s$.
    Thus $s = w \in W$, since $a\in\cocover{W} = \max(\ds W\setminus W)$, which again contradicts the fact that $s \notin W$.

    \smallskip

    \noindent
    ($\Leftarrow$) By assumption, we have $a\in \ds W\setminus W$, so there exists some $p\in \max(\ds W\setminus W) = \cocover W$, with $a\leq p$.
    Since $p \notin W$, we have $p \in \usminus{a}{W}$, and since $p\in \ds W \setminus W$, there exists $w \in W$ with $p \leq w$.
    If $a \neq p$, then $p \in \left(\usminus{a}{W}\right) \setminus a$, and thus $p \leq w$ would be a zigzag path between an element of $(\usminus{a}{W}) \setminus a$ and an element of $W$, which contradicts the assumption that $W$ is a connected component.
    It follows that $a = p$, which is a cocover of $W$, by assumption.
\end{proof}

\begin{eg}
    The spread $W = S\setminus \iprod{a}{s}$ in \cref{figure: spread lemmas} gives an example of \cref{lem: cocover vs connected}.
    The element $a \in \min S$ is a cocover of $W$.
    We have $\iprod{a}{s} = \usminus{a}{W}$ and $S = W \sqcup \usminus{a}{W}$.
    Note that every zigzag in $S$ from $u$ to $s$ must pass through $a$.
    Thus $W$ is a connected component of $S\setminus a$.
\end{eg}

\bibliography{bib}
\bibliographystyle{alpha}

\end{document}